\documentclass[12pt]{article}
\usepackage{amsmath}
\usepackage{latexsym}
\usepackage{amssymb}
\newtheorem{thm}{Theorem}[section]
\newtheorem{la}[thm]{Lemma}
\newtheorem{Defn}[thm]{Definition}
\newtheorem{Remark}[thm]{Remark}
\newtheorem{Conj}[thm]{Conjecture}
\newtheorem{prop}[thm]{Proposition}
\newtheorem{cor}[thm]{Corollary}
\newtheorem{Example}[thm]{Example}
\newtheorem{Number}[thm]{\!\!}
\newenvironment{defn}{\begin{Defn}\rm}{\end{Defn}}
\newenvironment{example}{\begin{Example}\rm}{\end{Example}}
\newenvironment{rem}{\begin{Remark}\rm}{\end{Remark}}

\newenvironment{numba}{\begin{Number}\rm}{\end{Number}}
\newenvironment{proof}{{\noindent\bf Proof.}}%
                  {\nopagebreak\hspace*{\fill}$\Box$\medskip\par}
\newcommand{\Punkt}{\nopagebreak\hspace*{\fill}$\Box$}
\newcommand{\wb}{\overline}

\newcommand{\tensor}{\otimes}
\newcommand{\mto}{\mapsto}
\newcommand{\isom}{\cong}
\newcommand{\N}{{\mathbb N}}
\newcommand{\B}{{\mathbb B}}
\newcommand{\bO}{{\mathbb O}}
\newcommand{\K}{{\mathbb K}}
\newcommand{\F}{{\mathbb F}}

\newcommand{\bL}{{\mathbb L}}

\newcommand{\Q}{{\mathbb Q}}
\newcommand{\Z}{{\mathbb Z}}

\newcommand{\cA}{{\cal A}}
\newcommand{\cV}{{\cal V}}

\newcommand{\ch}{{\mathfrak h}}

\newcommand{\one}{{\bf 1}}
\newcommand{\sub}{\subseteq}
\newcommand{\sbull}{{\scriptscriptstyle \bullet}}
\newcommand{\GL}{\operatorname{GL}}
\newcommand{\SL}{\operatorname{SL}}
\newcommand{\Iso}{\operatorname{Iso}}

\newcommand{\id}{\operatorname{id}}
\newcommand{\Spann}{\operatorname{span}}
\newcommand{\tor}{\operatorname{tor}}

\newcommand{\car}{\operatorname{char}}
\newcommand{\op}{\operatorname{op}}
\newcommand{\dt}{\operatorname{det}}
\begin{document}
%
%
%
%
\begin{center}
{\Large\bf Lectures on Lie Groups over Local Fields}\\[7mm]
{\bf Helge Gl\"{o}ckner\footnote{The author was
supported by the
German Research Foundation (DFG),
projects GL 357/2-1, GL 357/5-1
and GL 357/6-1.}}\vspace{2mm}
\end{center}
\begin{abstract}
\hspace*{-6mm}The goal of these notes is to
provide an introduction
to $p$-adic Lie groups
and Lie groups over fields
of Laurent series,
with an emphasis
on the dynamics of automorphisms
and the specialization of Willis' theory
to this setting.
In particular, we shall
discuss the scale,
tidy subgroups
and contraction groups
for automorphisms of
Lie groups over local fields. Special attention is paid to the case
of Lie groups over local fields of positive characteristic.\vspace{2mm}
\end{abstract}
{\bf Classification:}
Primary 22E20; Secondary 22D05, 22E35, 26E30, 37D10.\\[2mm]
\section*{Introduction}
Lie groups over local fields
(notably $p$-adic Lie groups)
are among those totally disconnected,
locally compact groups which
are both well accessible
and arise
frequently.
For example,
$p$-adic Lie groups play an important role
in the theory of pro-$p$-groups
(i.e.,
projective limits of
finite $p$-groups),
where they are called
``analytic pro-$p$-groups.''
In ground-breaking work in the 1960s,
Michel Lazard
obtained deep insights into the structure
of analytic pro-$p$-groups
and characterized them
within the class of pro-$p$-groups~\cite{Laz}
(see \cite{Dix} and \cite{HOR} for later developments).\\[2.1mm]
It is possible to study Lie groups over local
fields from various points of view and
on various levels,
taking more and more structure
into account.
At the most
basic level,
they can be considered as mere topological
groups.
Next, we can consider them as analytic manifolds,
enabling us to use ideas from Lie theory
and properties of analytic functions.
At the highest level,
one can focus on those Lie groups
which are linear
algebraic groups over local fields
(in particular, reductive or semi-simple groups),
and study them using techniques from
the theory of algebraic groups
and algebraic geometry.\\[2.1mm]
While much of the literature
focusses either on pro-$p$-groups
or algebraic
groups, our
aims are complementary:
We emphasize
aspects at the borderline
between Lie theory
and the structure theory of
totally disconnected,
locally compact groups,
as developed in \cite{Wil},
\cite{Wi2}
and \cite{BaW}.
In particular,
we shall discuss the scale,
tidy subgroups and
contraction groups
for automorphisms of
Lie groups over local fields.\\[2.1mm]
For each of these
topics, an algebraic
group structure does not play a role,
but merely the Lie group structure.\\[2.1mm]
{\bf Scope and structure of the lectures.}
After an introduction
to some essentials of Lie
theory
and calculus over local fields,
we discuss
the scale, contraction groups,
tidy subgroups (and related topics)
in three stages:
\begin{itemize}
\item
In a first step,
we consider linear automorphisms
of finite-dimensional
vector spaces over local
fields.
\end{itemize}
A sound understanding of this
special case is essential.
\begin{itemize}
\item
Next,
we turn to automorphisms
of $p$-adic Lie groups.
\item
Finally, we discuss Lie groups
over local fields of positive
characteristic and their automorphisms.
\end{itemize}
For various reasons,
the case of positive characteristic
is more complicated than the $p$-adic
case, and it is one of our goals
to explain which additional ideas
are needed to tackle this case.\\[2.1mm]
As usual in an expository
article, we shall present many facts
without proof.
However, we have taken care to
give at least a sketch of proof for all
central results, and to explain the
main ideas.
In particular, we have taken care to explain
carefully the ideas needed to deal with Lie groups over
local fields of positive characteristic,
and found it appropriate to
give slightly more details when they are concerned
(the more so because
not all of the results
are available yet in the published literature).\\[2.1mm]
Mention should be made of what these
lectures do not strive to achieve.\\[2.1mm]
First of all, we do not intend
to give an introduction to linear
algebraic groups over local fields,
nor to Bruhat-Tits buildings or to the
representation theory
of $p$-adic groups
(and harmonic analysis thereon).
The reader will find nothing
about these topics here.\\[2.1mm]
Second, we shall hardly speak
about the theory of analytic
pro-$p$-groups,
although this theory is certainly located at a very similar position
in the
spectrum
ranging from topological groups to algebraic groups
(one of Lazard's results will be recalled
in Section~\ref{sec3},
but no later developments).
One reason is that we want to focus
on aspects related to the structure
theory of totally disconnected, locally compact groups---which
is designed primarily for the
study of non-compact groups.\\[2.1mm]
It is interesting to note that
also in the area of pro-$p$-groups,
Lie groups
over fields (and suitable pro-$p$-rings)
of positive characteristic are
attracting more
and more attention.
The reader is referred to the seminal
work~\cite{LaS}
and subsequent research
(like \cite{ANS}, \cite{BaL}, \cite{CdS},
\cite{JZ} and \cite{JZK}; cf.\ also \cite{HOR}).\\[2.1mm]
Despite the importance of $p$-adic Lie groups
in the area of pro-$p$-groups,
it is not fully clear yet whether
$p$-adic Lie groups (or Lie groups over local fields)
play a fundamental role in the theory
of general
locally compact groups
(comparable to that
of real Lie groups, as in \cite{HaM}).
Two recent studies indicate that
this may be the case,
at least in the presence of group actions.
They describe situations
where Lie groups over local fields
are among the building blocks for general structures:
\begin{description}
\item[Contraction groups]
Let $G$ be a totally disconnected,
locally compact group
and $\alpha\colon G\to G$ be an
automorphism
which is \emph{contractive},
i.e., $\alpha^n(x)\to 1$ as $n\to\infty$,
for each $x\in G$ (where $1\in G$ is the neutral element).
Then the torsion elements form a closed
subgroup $\tor(G)$ and
\[
G\; =\; \tor(G)\times G_{p_1}\times\cdots \times G_{p_n}
\]
for certain $\alpha$-stable $p$-adic Lie groups $G_p$
(see~\cite{SIM}).
\item[{\boldmath $\Z^n$}-actions]
If $\Z^n$ acts with a dense orbit
on a locally compact group~$G$ via automorphisms,
then $G$ has a compact normal subgroup~$K$ invariant
under the action
such that
\[
G/K\; \isom \; \bL_1\times\cdots\times \bL_m\, ,
\]
where $\bL_1,\ldots, \bL_m$ are local fields
of characteristic~$0$
and the action is diagonal,
via scalar multiplication
by field elements~\cite{DSW}.
\end{description}
In both studies, Lazard's
theory of analytic pro-$p$-groups
accounts for the occurrence of $p$-adic Lie groups.\\[5mm]
{\bf Overview of the lectures}\\[2mm]
1.\hspace*{4mm}Generalities\,\dotfill\,\pageref{sec1}\\[1.5mm]
2.\hspace*{4mm}Basic facts concerning $p$-adic
Lie groups\,\dotfill\,\pageref{sec2}\\[1.5mm]
3.\hspace*{4mm}Lazard's characterization of $p$-adic
Lie groups\,\dotfill\,\pageref{sec3}\\[1.5mm]
4.\hspace*{4mm}Iteration of linear automorphisms\,\dotfill\,\pageref{secitlin}\\[1.5mm]
5.\hspace*{4mm}Scale and tidy subgroups for automorphisms
of $p$-adic Lie groups\,\dotfill\,\pageref{sec5}\\[1.5mm]
6.\hspace*{4mm}$p$-adic contraction groups\,\dotfill\,\pageref{sec6}\\[1.5mm]
7.\hspace*{3.9mm}Pathologies in positive
characteristic\,\dotfill\,\pageref{secpatho}\\[1.5mm]
8.\hspace*{3.6mm}Tools from non-linear analysis:
invariant manifolds\,\dotfill\,\pageref{sec9}\\[1.5mm]
9.\hspace*{3.6mm}The scale, tidy subgroups and contraction groups
in positive\\
\hspace*{6.6mm}characteristic\,\dotfill\,\pageref{sec10}\\[1.5mm]
10.\hspace*{1.7mm}The structure of contraction groups in the
case of positive\\
\hspace*{6.6mm}characteristic\,\dotfill\,\pageref{sec11}\\[1.5mm]
11.\hspace*{1.6mm}Further generalizations to positive
characteristic\,\dotfill\,\pageref{sec12}\vspace{2mm}
\section{Generalities}\label{sec1}
This section compiles
elementary definitions and facts concerning local fields,\linebreak
analytic functions and Lie groups.\\[2.1mm]
Basic information on
local fields can be found in many books,
e.g.\ \cite{Wei} and \cite{Sch}.
Our main sources for Lie groups over local
fields are~\cite{Ser}
and~\cite{Bo2}.
\subsection*{{\normalsize Local fields}}\label{locfds}
\noindent
By a \emph{local field}, we mean
a totally disconnected, locally compact, non-discrete
topological field~$\K$.
Each local field admits an \emph{ultrametric
absolute value} $|.|$ defining
its topology, i.e.,
\begin{itemize}
\item[(a)]
$|t|\geq 0$ for each $t\in \K$, with equality if and only if
$t=0$;
\item[(b)]
$|st|=|s|\cdot |t|$ for all $s,t\in \K$;
\item[(c)]
The \emph{ultrametric inequality} holds, i.e.,
$|s+t|\leq \max\{|s|,|t|\}$ for all
$s,t\in \K$.
\end{itemize}
An example of such an absolute value
is what we call the \emph{natural absolute
value}, given by $|0|:=0$ and
\begin{equation}\label{dffrmnatu}
|x|\; :=\; \Delta_\K(m_x)
\quad \mbox{for $\,x\in \K\setminus\{0\}$}
\end{equation}
(cf.\ \cite[Chapter~II, \S2]{Wei}),
where $m_x\colon \K\to\K$, $y\mto xy$
is scalar multiplication by~$x$ and $\Delta_\K(m_x)$
its module
(the definition of which is recalled
in~\ref{recamodu}).\footnote{Note that if $\K$ is an extension of $\Q_p$
of degree~$d$, then $|p|=p^{-d}$
depends on the extension.}\\[2.1mm]
It is known that every local field~$\K$
either is
a field of formal Laurent
series over some finite field
(if $\car(\K)>0$),
or a finite extension
of the field of $p$-adic numbers for some
prime~$p$ (if $\car(\K)=0$).
Let us fix our notation concerning these
basic examples.
\begin{example}
Given a prime number $p$,
the field $\Q_p$ of $p$-adic
numbers is the completion
of $\Q$ with respect to
the $p$-adic absolute value,
\[
\left| p^k \frac{n}{m}\right|_p\;:=\;p^{-k}\quad
\mbox{for $k\in \Z$ and $n,m\in \Z\setminus p\Z$.}
\]
We use the same notation,
$|.|_p$, for the extension of the $p$-adic
absolute value to~$\Q_p$.
Then the topology coming from $|.|_p$ makes
$\Q_p$ a local field,
and $|.|_p$ is the natural absolute value
on~$\Q_p$.
Every non-zero element $x$ in $\Q_p$ can be written
uniquely in the form
\[
x\;=\; \sum_{k=n}^\infty a_k\, p^k
\]
with $n\in \Z$, $a_k\in \{0,1,\ldots, p-1\}$
and $a_n\not=0$.
Then $|x|_p=p^{-n}$.
The elements
of the form $\sum_{k=0}^\infty a_kp^k$
form the subring
$\Z_p=\{x\in \Q_p \colon |x|_p\leq 1\}$
of~$\Q_p$, which is open
and also compact,
because it is homeomorphic to
$\{0,1,\ldots, p-1\}^{\N_0}$
via $\sum_{k=0}^\infty a_kp^k\mto (a_k)_{k\in \N_0}$.
\end{example}
\begin{example}
Given a finite field~$\F$ (with~$q$ elements),
we let $\F(\!(X)\!)\sub \F^\Z$
be the field of formal Laurent
series
$\sum_{k=n}^\infty a_kX^k$
with $a_k\in \F$. 
Here addition is
pointwise, and multiplication
is given by the Cauchy product.
We endow $\F(\!(X)\!)$
with the topology arising from the ultrametric absolute value
\begin{equation}\label{nather}
\left|\sum_{k=n}^\infty a_kX^k\right|
\; :=\; q^{-n}\quad\mbox{if $\,a_n\not=0$.}
\end{equation}
Then the set $\F[\hspace*{-.15mm}[X]\hspace*{-.17mm}]$
of formal power series
$\sum_{k=0}^\infty a_kX^k$
is a compact and open subring of
$\F(\!(X)\!)$, and thus
$\F(\!(X)\!)$
is a local field.
Its natural absolute value
is given by~(\ref{nather}).
\end{example}
\noindent
Beyond local fields,
we shall occasionally
use \emph{ultrametric fields}
$(\K,|.|)$.
Thus $\K$ is a field
and $|.|$ an ultrametric absolute value
on~$\K$ which defines a non-discrete
topology on~$\K$.
For example, we occasionally
use an algebraic closure
$\wb{\K}$ of a local field~$\K$
and exploit that an ultrametric absolute value~$|.|$
on~$\K$ extends uniquely
to an ultrametric absolute value on~$\wb{\K}$
(see, e.g., \cite[Theorem~16.1]{Sch}).
The same notation, $|.|$,
will be used for the extended absolute value.
An ultrametric field $(\K,|.|)$
is called \emph{complete} if $\K$ is a complete
metric space with respect to the
metric given by $d(x,y):=|x-y|$.
\subsection*{{\normalsize Basic consequences of
the ultrametric inequality}}\label{bscconsultra}
Let $(\K,|.|)$ be an ultrametric field
and $(E,\|.\|)$ be a normed $\K$-vector
space whose norm is ultrametric in the sense that
$\|x+y\|\leq \max\{\|x\|,\|y\| \}$ for all
$x,y\in E$.\\[2.1mm]
Since $\|x\|=\|x+y-y\| \leq \max\{\|x+y\|,\|y\|\}$,
it follows that $\|x+y\|\geq \|x\|$ if $\|y\|<\|x\|$
and hence
\begin{equation}\label{winner}
\|x+y\|\;=\;\|x\|\quad \mbox{for all $x,y\in E$ such that
$\|y\|<\|x\|$}
\end{equation}
(``the strongest wins,'' \cite[p.\,77]{Rbe}).
Hence, small perturbations
have much smaller impact in the ultrametric
case than they would have in the real
case (a philosophy which will
be made concrete below).\\[2.1mm]
The ultrametric inequality
has many other useful consequences.
For example, consider the balls
\[
B_r^E(x)\; :=\;
\{y\in E\colon \|y-x\|<r\}
\]
and
\[
\wb{B}_r^E(x)\; :=\;
\{y\in E\colon \|y-x\|\leq r\}
\]
for $x\in E$, $r\in \;]0,\infty[$.
Then $B_r^E(0)$ and $\wb{B}^E_r(0)$
are \emph{subgroups} of $(E,+)$
with non-empty interior
(and hence are both open and closed).
Specializing to $E=\K$, we see that
\begin{equation}\label{dfnvalr}
\bO\; :=\; \{t\in \K\colon |t|\leq 1\}
\end{equation}
is an open subring of~$\K$,
its so-called \emph{valuation ring}.
If $\K$ is a local field,
then $\bO$ is a compact
subring of~$\K$ (which is maximal
and hence independent of the choice of
absolute value).
\subsection*{Calculation of indices}\label{calcind}
Indices of compact
subgroups inside others are omni\-present
in the structure theory
of totally disconnected groups.
We therefore
perform
some basic calculations
of indices now.
Haar measure $\lambda_G$
on a locally compact group~$G$
will be used as a tool
and also the notion of the
\emph{module} of an automorphism,
which measures the distortion
of Haar measure.
We recall:
\begin{numba}\label{recamodu}
Let $G$ be a locally compact group
and $\B(G)$ be the $\sigma$-algebra
of Borel subsets of~$G$.
We let $\lambda_G\colon \B(G)\to [0,\infty]$
be a Haar measure on~$G$, i.e.,
a non-zero Radon measure which is left-invariant
in the sense that $\lambda_G(gA)=\lambda_G(A)$ for all
$g\in G$ and $A\in \B(G)$.
It is well-known that a Haar measure always exists,
and that it is unique up to multiplication with a positive
real number (cf.\ \cite{HaR}).
If $\alpha\colon G\to G$ is a (bicontinuous)
automorphism, then also
\[
\B(G)\to [0,\infty],\qquad
A\mto \lambda_G(\alpha(A))
\]
is a left-invariant non-zero Radon measure
on~$G$ and hence a multiple of Haar measure:
There exists $\Delta(\alpha)>0$ such that
$\lambda_G(\alpha(A))=\Delta(\alpha)\lambda_G(A)$
for all $A\in \B(G)$.
If $U\sub G$ is a relatively compact,
open, non-empty subset,
then
\begin{equation}\label{calcmodul}
\Delta(\alpha)=\frac{\lambda_G(\alpha(U))}{\lambda_G(U)}
\end{equation}
(cf.\ \cite[(15.26)]{HaR}, where however
the conventions differ). We also write $\Delta_G(\alpha)$
instead of $\Delta(\alpha)$, if we wish to emphasize
the underlying group~$G$.
\end{numba}
\begin{rem}\label{indcoop}
Let $U$ be a compact open
subgroup of $G$. If $U\sub \alpha(U)$,
with index $[\alpha(U):U]=:n$,
we can pick representatives $g_1,\ldots, g_n\in \alpha(U)$
for the left cosets of $U$ in~$\alpha(U)$.
Exploiting the left-invariance  of Haar measure,
(\ref{calcmodul})
turns into
\begin{equation}\label{howgtmod}
\Delta(\alpha)\, =\, \frac{\lambda_G(\alpha(U))}{\lambda_G(U)}
\, =\, \sum_{j=1}^n\frac{\lambda_G(g_jU)}{\lambda_G(U)}
\, =\, [\alpha(U):U]\,.
\end{equation}
If $\alpha(U)\sub U$, applying (\ref{howgtmod})
to $\alpha^{-1}$ instead of~$\alpha$
and $\alpha(U)$ instead of~$U$, we obtain
\begin{equation}\label{shrinkindex}
\Delta(\alpha^{-1})
\, =\, [U:\alpha(U)]\,.
\end{equation}
\end{rem}
\noindent
In the following,
the group $\GL_n(\K)$ of invertible $n\times n$-matrices
will frequently be identified with the group $\GL(\K^n)$ of linear
automorphisms of~$\K^n$.
\begin{la}\label{laindx}
Let $\K$ be a local field,
$|.|$
its natural absolute value, 
and~$A \in \GL_n(\K)$.
Then
\begin{equation}\label{modvia}
\Delta_{\K^n}(A)\;=\; |\hspace*{-.4mm}\dt A\hspace*{.4mm}|
\;=\;\prod_{i=1}^n |\lambda_i|\,,
\end{equation}
where $\lambda_1,\ldots,\lambda_n$
are the eigenvalues of~$A$
in an algebraic closure $\wb{\K}$
of~$\K$.
\end{la}
\begin{proof}
Let $\lambda_{\K^n}$ and $\lambda_\K$ be Haar measures
on $(\K^n,+)$ and $(\K,+)$, respectively, such that
$\lambda_{\K^n}=\lambda_\K\otimes\cdots\otimes \lambda_\K$.
Let $\bO\sub \K$ be the valuation ring
and $U:=\bO^n$.
Since both $A\mto \Delta_{\K^n}(A)$ and
$A\mto |\hspace*{-.4mm}\dt A\hspace*{.4mm}|$ are homomorphisms
from $\GL_n(\K)$ to the multiplicative group $]0,\infty[$,
it suffices to check the first
equality in (\ref{modvia})
for diagonal
matrices and elementary matrices
of the form $\one+e_{ij}$ with
matrix units $e_{ij}$,
where $i\not=j$
(because these generate $\GL_n(\K)$ as a group).
For such an elementary matrix~$A$, we have
$A(U)=U$ and thus
\[
\Delta_{\K^n}(A)\;=\;
\frac{\lambda_{\K^n}(A(U))}{\lambda_{\K^n}(U)}
\;=\;1\;=\;
|\hspace*{-.4mm}\dt A\hspace*{.4mm}|\,.
\]
If $A$ is diagonal with diagonal entries $t_1,\ldots, t_n$,
we also have
\[
\Delta_{\K^n}(A)\;=\;
\frac{\lambda_{\K^n}(A(U))}{\lambda_{\K^n}(U)}
\;=\;\prod_{i=1}^n \frac{\lambda_\K(t_i \bO)}{\lambda_\K(\bO)}
\;=\; \prod_{i=1}^n \Delta_\K(m_{t_i})
\;=\; \prod_{i=1}^n |t_i|
\;=\;
|\hspace*{-.4mm}\dt A\hspace*{.4mm}|\,,
\]
as required.
The final equality in~(\ref{modvia})
is clear if all eigenvalues lie in~$\K$.
For the general case, pick
a finite extension~$\bL$ of~$\K$ containing
the eigenvalues,
and let $d:=[\bL:\K]$
be the degree of the field extension.
Then $\Delta_{\bL^n}(A)=(\Delta_{\K^n}(A))^d$.
Since the extended
absolute value is given
by
\[
|x|\; =\; \sqrt[d]{\Delta_{\bL}(m_x)}
\]
(see \cite[Chapter 9, Theorem~9.8]{Jcb}
or \cite[Exercise 15.E]{Sch}),
the final equality follows
from the special case
already treated (applied now to~$\bL$).
\end{proof}
\begin{numba}\label{NEWball}
For later use,
consider a ball $B_r(0)\sub \Q_p^n$
with respect to the maximum norm,
where $r\in \;]0,\infty[$.
Let $m\in \N$ and $A\in\GL_n(\Q_p)$
be the diagonal matrix with diagonal entries
$p^m , \ldots  ,\hspace*{.4mm}p^m$.
Then (\ref{shrinkindex})
and (\ref{modvia}) imply that
\begin{equation}\label{goodcf}
[B_r(0):B_{p^{-m}r}(0)]\; =\;
[B_r(0):A.B_r(0)]
\; =\; |\hspace*{-.3mm}\dt(A^{-1})\hspace*{.4mm}|\; =\; p^{mn}\,.
\end{equation}
\end{numba}
\subsection*{Analytic functions,
manifolds and Lie groups}\label{anft}
Given a local field $(\K,|.|)$
and $n\in \N$,
we equip $\K^n$
with the maximum norm,
$\|(x_1,\ldots, x_n)\|:=\max\{|x_1|,\ldots,|x_n|\}$
(the choice of norm does not really matter
because all norms are equivalent; see \cite[Theorem~13.3]{Sch}).
If $\alpha\in \N_0^n$ is a multi-index, 
we write $|\alpha|:=\alpha_1+\cdots+\alpha_n$.
We mention that confusion with the absolute value $|.|$ is
unlikely; the intended
meaning of~$|.|$ will always be clear from the context.
If $\alpha\in \N_0^n$ and $y=(y_1,\ldots, y_n)\in \K^n$,
we abbreviate $y^\alpha:=y_1^{\alpha_1}\cdots \hspace*{.3mm}y_n^{\alpha_n}$,
as usual. See \cite{Ser} for the following concepts.
\begin{defn}
Given an open subset
$U\sub \K^n$, a map $f\colon U\to \K^m$
is called
\emph{analytic}\footnote{In other parts of the literature
related to rigid analytic geometry,
such functions are called \emph{locally analytic}
to distinguish them from functions which are globally given
by a power series.}
if it is given locally by a convergent
power series around each point $x\in U$, i.e.,
\[
f(x+y)=\sum_{\alpha\in \N_0^n} a_\alpha \, y^\alpha
\quad
\mbox{for all $\,y\in \wb{B}^{\K^n}_r\hspace*{-1mm}(0)$,}
\]
with $a_\alpha\in \K^m$
and some $r>0$ such that
$\wb{B}^{\K^n}_r\hspace*{-1mm}(x)\sub U$ and
\[
\sum_{\alpha\in \N_0^n} \|a_\alpha\|\,
r^{|\alpha|}\;<\;\infty\,.
\]
\end{defn}
\noindent
It can be shown that compositions
of analytic functions are analytic
\cite[Theorem, p.\,70]{Ser}.
We can therefore define an $n$-dimensional
analytic manifold~$M$
over a local field~$\K$
in the usual
way, namely as a Hausdorff
topological space~$M$,
equipped with a set~$\cA$ of homeomorphisms
$\phi$ from open subsets of~$M$ onto
open subsets of $\K^n$
such that the transition map
$\psi\circ \phi^{-1}$
is analytic, for all $\phi,\psi\in \cA$.\\[2.1mm]
Analytic mappings between
analytic manifolds are defined
as usual
(by checking analyticity
in local charts).\\[2.1mm]
A \emph{Lie group}
over a local field~$\K$ is a group~$G$,
equipped with a (finite-dimensional)
analytic manifold structure which turns
the group multiplication
\[
m \colon G\times G\to G\,,\quad
m(x,y):=xy
\]
and the group inversion
\[
j \colon G\to G\,,\quad
j(x):=x^{-1}
\]
into analytic mappings.\\[2.1mm]
Besides the additive groups of finite-dimensional
$\K$-vector spaces,
the most obvious examples of $\K$-analytic Lie
groups are general linear groups.
\begin{example}
$\GL_n(\K)=\dt^{-1}(\K^\times)$ is an open subset
of the space $M_n(\K)\isom \K^{n^2}$
of $n\times n$-matrices
and hence is an $n^2$-dimensional
$\K$-analytic manifold.
The group operations are
rational maps and hence analytic.
\end{example}
\noindent
More generally,
one can show (cf.\ \cite[Chapter~I, Proposition~2.5.2]{Mar}):
\begin{example}
Every (group of $\K$-rational points of a)
linear algebraic group defined over~$\K$
is a $\K$-analytic Lie group,
viz.\ every subgroup
$G\leq \GL_n(\K)$ which is the set of
joint zeros of a set of polynomial
functions $M_n(\K)\to\K$.
For example, $\SL_n(\K)=\{ A \in \GL_n(\K)\colon
\dt(A )=1\}$
is a $\K$-analytic Lie group.
\end{example}
\begin{rem}
There are much more general
Lie groups than linear Lie groups.
For instance,
in Remark~\ref{nonlinea}
we shall
encounter an analytic Lie group~$G$
over $\K=\F_p(\!(X)\!)$
which does not admit
a faithful, continuous
linear representation
$G\to\GL_n(\K)$ for any~$n$.
\end{rem}
\noindent
{\bf The Lie algebra functor.}
If $G$ is a $\K$-analytic Lie group,
then its tangent space $L(G):=T_1(G)$
at the identity element can be made
a Lie algebra via the identification
of $x\in L(G)$ with the corresponding
left invariant vector field on~$G$
(noting that the left invariant
vector fields form a Lie subalgebra
of the Lie algebra $\cV^\omega(G)$
of all analytic vector fields on~$G$).\\[2.1mm]
If $\alpha \colon G\to H$ is an analytic
group homomorphism between $\K$-analytic Lie groups,
then the tangent map $L(\alpha):=T_1(\alpha)\colon L(G)\to L(H)$
is a linear map and actually a Lie algebra
homomorphism (cf.\ \cite[Chapter~3, \S3.7 and \S3.8]{Bo2},
Lemma~5.1 on p.\,129 in
\cite[Part II, Chapter~V.1]{Ser}).
An \emph{analytic automorphism} of a Lie group $G$ is
an invertible group homomorphism $\alpha\colon G\to G$
such that both $\alpha$ and $\alpha^{-1}$ are analytic.
\subsection*{Ultrametric inverse function theorem}\label{ultinv}
Since small perturbations
do not change the size of a given non-zero vector
in the ultrametric case
(as ``the strongest wins''),
the ultrametric inverse function theorem
has a much nicer form
than its classical real counterpart.
Around a point
with invertible differential,
an analytic map
behaves on balls simply like
an affine-linear map
(namely its linearization).\\[2mm]
In the following three propositions,
we let $(\K,|.|)$ be a complete
ultrametric field and equip $\K^n$ with any
ultrametric norm (e.g.,
the maximum norm).
Given $x\in \K^n$ and $r>0$, we abbreviate $B_r(x):=B^{\K^n}_r(x)$.
The total differential of $f$ at~$x$
is denoted by~$f'(x)$.
Now the ultrametric inverse function theorem
(for analytic functions)
reads as follows:\footnote{A proof is obtained, e.g.,
by combining \cite[Proposition~7.1\,(b)$'$]{IMP}
(a $C^k$-analogue of Proposition~\ref{invfct})
with the version of
the inverse function theorem for
analytic maps from \cite[p.\,73]{Ser}.}
\begin{prop}\label{invfct}
Let $f\colon U\to \K^n$
be an analytic map on an open set $U\sub \K^n$
and $x\in U$
such that $f'(x)\in \GL_n(\K)$.
Then there exists $r>0$ such that
$B_r(x)\sub U$,
\begin{equation}\label{great}
f(B_s(y))\;=\;f(y)+f'(x).B_s(0)\quad
\mbox{for all $y\in B_r(x)$ and $s\in \;]0,r]$,}
\end{equation}
and $f|_{B_s(y)}$ is an analytic diffeomorphism
onto its open image.
If $f'(x)$ is an isometry,
then so is $f|_{B_r(x)}$ for small~$r$.\,\Punkt
\end{prop}
\noindent
It is useful that $r$ can be chosen
uniformly in the presence of parameters.
As a special case of \cite[Theorem~8.1\,(b)$'$]{IMP} (which only
requires that $f$ be $C^k$ in a suitable sense),
an `ultrametric
inverse function theorem with parameters'
is available:\footnote{Combining the cited theorem
and Proposition~\ref{invfct},
one can also show that,
for $Q$ and $r$ small enough,
$f_q|_{B_r(x)}\colon B_r(x)\to f_q(x)+f'_p(x).B_r(0)=f_p(x)+f'_p(0).B_r(0)$
is an analytic diffeomorphism
for each $q\in Q$.
However, this additional information
shall not be used
in the current work. We also mention that the results from \cite{IMP}
were strengthened
further in~\cite{IM2}.}
\begin{prop}\label{invpar}
Let $P\sub \K^m$ and $U\sub \K^n$ be open,
$f\colon P \times U\to \K^n$
be a $\K$-analytic map, $p\in P$
and $x\in U$
such that $f'_p(x)\in \GL_n(\K)$,
where $f_p:=f(p,\sbull)\colon U\to\K^n$.
Then there exists an open neighbourhood $Q\sub P$ of~$p$
and $r>0$ such that
$B_r(x)\sub U$ and
\begin{equation}\label{great2}
f_q(B_s(y))\;=\;f_q(y)+f'_p(x).B_s(0)
\end{equation}
for all $q\in Q$, $y\in B_r(x)$ and $s\in \;]0,r]$.
If $f'_p(x)$ is an isometry,
then also $f_q|_{B_r(x)}$ is an isometry
for all $q\in Q$,
if $Q$ and~$r$ are chosen sufficiently small.\,\Punkt
\end{prop}
\noindent
We mention that the group of linear isometries
is large (an open subgroup).
\begin{prop}\label{propisos}
The group $\Iso(\K^n)$
of linear isometries
is open in $\GL_n(\K)$.
If $\K^n$ is equipped with the maximum norm,
then
\[
\Iso(\K^n)\;=\; \GL_n(\bO)\; =\; \{ A\in \GL_n(\K)\colon
A,A^{-1}\in M_n(\bO)\}\,,
\]
where $\bO$ $($as in {\rm(\ref{dfnvalr}))}
is the valuation ring of~$\K$.
\end{prop}
\begin{proof}
The subgroup $\Iso(\K^n)$ will be open in $\GL_n(\K)$
if it is an identity neighbourhood.
The latter is guaranteed if if we can prove that
${\bf 1}+A\in \Iso(\K^n)$ for all $A\in M_n(\K)$
of operator norm
\[
\|A\|_{\op}\, :=\, \sup\Big\{ \frac{\|Ax\|}{\|x\|}\colon 0\not= x\in \K^n\Big\}\,
<\, 1\, .
\]
However, for each $0\not=x\in \K^n$
we have $\|Ax\|\leq \|A\|_{\op}\|x\|<\|x\|$
and hence $\|({\bf 1}+A)x\|=\|x+Ax\|=\|x\|$
by (\ref{winner}), as the strongest wins.

Now assume that $\|.\|$ is the maximum norm
on $\K^n$, and $A\in \GL_n(\K)$.
If $A$ is an isometry, then $\|A e_i\|=1$
for the standard basis vector $e_i$ with
$j$-th component $\delta_{ij}$,
and hence $A\in M_n(\bO)$. Likewise,
$A^{-1}\in M_n(\bO)$ and hence $A\in \GL_n(\bO)$.
If $A\in \GL_n(\bO)$, then
$\|A x\|\leq \|x\|$
and $\|x\|=\|A^{-1}Ax\|\leq\|Ax\|$,
whence $\|Ax\|=\|x\|$
and $A$ is an isometry.
\end{proof}
\subsection*{Construction of small open
subgroups}\label{constsmall}
Let $G$ be a Lie group over a local field~$\K$,
and $|.|$ be an absolute value on~$\K$ defining its topology.
Fix an ultrametric norm $\|.\|$ on~$L(G)$
and abbreviate $B_r(x):=B^{L(G)}_r(x)$
for $x\in L(G)$ and $r>0$.
Using a chart $\phi\colon G\supseteq U\to V\sub L(G)$
around~$1$ such that $\phi(1)=0$,
the group multiplication gives rise
to a multiplication $\mu\colon W\times W\to V$,
$\mu(x,y)=x*y$ for an open $0$-neighbourhood
$W\sub V$, via
\[
x*y\; :=\; \phi(\phi^{-1}(x)\phi^{-1}(y))\,.
\]
It is easy to see that the first order Taylor
expansions of multiplication and inversion
in local coordinates read
\begin{equation}\label{Tay01}
x*y\;=\; x+y+\cdots
\end{equation}
and
\begin{equation}\label{Tay02}
x^{-1}\;=\; -x+\cdots
\end{equation}
(compare \cite[p.\,113]{Ser}).
Applying the ultrametric inverse function
theorem with\linebreak
parameters
to the maps
$(x,y)\mto x*y$
and $(x,y)\mto y*x$,
we find $r>0$ with $B_r(0)\sub W$ such that
\begin{equation}\label{premeas}
x*B_s(0)\;=\; x+B_s(0)\;=\; B_s(0)*x
\end{equation}
for all $x\in B_r(0)$ and $s\in \;]0,r]$
(exploiting that all relevant differentials
or partial differentials
are isometries
in view of (\ref{Tay01}) and
(\ref{Tay02})).
In particular, (\ref{premeas}) implies that
$B_s(0)*B_s(0)=B_s(0)$ for each $s\in \;]0,r]$,
and thus also
$y^{-1}\in B_s(0)$ for each $y\in B_s(0)$.\\[2.1mm]
Summing up:
\begin{la}\label{laballs}
$(B_s(0),*)$ is a group
for each $s\in \;]0,r]$ and hence
$\phi^{-1}(B_s(0))$ is a compact
open subgroup of~$G$, for each $s\in \;]0,r]$.
Moreover, $B_s(0)$ is a normal subgroup of $(B_r(0),*)$.\,\Punkt
\end{la}
\noindent
Thus small balls in $L(G)$ correspond to
compact open subgroups in~$G$.
These special subgroups are very useful
for many purposes. In particular, we shall
see later that for suitable choices of the norm $\|.\|$,
the groups $\phi^{-1}(B_s(0))$ will be tidy
for a given automorphism~$\alpha$,
as long as~$\alpha$ is well-behaved (exceptional
cases where this goes wrong will be pinpointed as well).
\begin{rem}\label{smidx}
Note that (\ref{premeas}) entails that
the indices of $B_s(0)$ in $(B_r(0),+)$
and $(B_r(0),*)$ coincide (as the cosets
coincide). This observation will be useful later.
\end{rem}
\section{Basic facts concerning
{\boldmath $p$}-adic Lie groups}\label{sec2}
For each local field $\K$ of characteristic~$0$,
the exponential series $\sum_{k=0}^\infty\frac{1}{k!}A^k$
converges on some $0$-neighbourhood
$U\sub M_n(\K)$ and defines an
analytic mapping
$\exp\colon U\to \GL_n(\K)$.
More generally, it can be shown
that every $\K$-analytic Lie group~$G$
admits an exponential function
in the following sense:
\begin{defn}
An analytic map $\exp_G\colon U\to G$
on an
open $0$-neighbour\-hood
$U\sub L(G)$ with $\bO U\sub U$
is called an \emph{exponential
function} if $\exp_G(0)=1$,
$T_0(\exp_G)=\id$
and
\[
\exp_G((s+t)x)\; =\; \exp_G(sx)\exp_G(tx)
\quad \mbox{for all $x\in U$ and $s,t\in \bO$,}
\]
where $\bO:=\{t\in \K\colon |t|\leq 1\}$
is the valuation ring.
\end{defn}
\noindent
Since $T_0(\exp_G)=\id$,
after shrinking~$U$ one can assume that
$\exp_G(U)$ is open and $\exp_G$ is a
diffeomorphism onto its image.
By Lemma~\ref{laballs},
after shrinking~$U$ further if necessary,
we may assume that $\exp_G(U)$
is a subgroup of~$G$.
Hence also~$U$
can be considered as a Lie group.
The Taylor expansion
of multiplication with respect to the
logarithmic chart $\exp_G^{-1}$ is given by the Baker-Campbell-Hausdorff
(BCH-) series
\begin{equation}\label{BCH}
x*y\, = \, x+y+\frac{1}{2}[x,y]+\cdots
\end{equation}
(all terms of which are nested
Lie brackets with rational coefficients),
and hence $x*y$ is given by this series
for small~$U$
(cf.\ Lemma~3 and Theorem~2 in \cite[Chapter~3, \S4, no.\,2]{Bo2}).
In this case, we call
$\exp_G(U)$ a \emph{BCH-subgroup of~$G$.}\\[2.1mm]
For later use, we note that
\begin{equation}\label{newdetai1}
x^n\;=\; x*\cdots * x\; =\; nx
\end{equation}
for all $x\in U$ and $n\in \N_0$
(since $[x,x]=0$ in (\ref{BCH})).
As a consequence, the closed subgroup of $(U,*)$
generated by $x\in U$ is of the form
\begin{equation}\label{newdetai2}
\wb{\langle x\rangle}\; =\; \Z_p\, x\,.
\end{equation}
We also note that
\begin{equation}\label{simplexp}
\exp_U\; :=\; \id_U
\end{equation}
is an exponential map for~$U$.\\[2.1mm]
Next, let us consider homomorphisms between Lie groups.
Recall that if
$\alpha \colon G\to H$ is an analytic homomorphism
between real Lie groups,
then the diagram
\[
\begin{array}{rcl}
G\,  & \stackrel{\alpha}{\longrightarrow} & H\\
\exp_G\uparrow \; & & \; \uparrow \exp_H\\
L(G) &\stackrel{L(\alpha)}{\longrightarrow} & L(H)
\end{array}
\]
commutes (a fact referred to
as the ``naturality of $\exp$'').
If $\alpha\colon G\to H$ is an analytic homomorphism
between Lie groups over a local field~$\K$
of characteristic~$0$,
we can still choose $\exp_G\colon U_G\to G$
and $\exp_H\colon U_H\to H$
with $L(\alpha).U_G \sub U_H$
and
\begin{equation}\label{locnat}
\exp_H\circ L(\alpha)|_{U_G}\; =\; \alpha \circ \exp_G
\end{equation}
(see Proposition~8 in
\cite[Chapter~3, \S4, no.\,4]{Bo2}).\\[2.1mm]
The following fact (see
\cite[Chapter 3, Theorem 1 in \S8, no.\,1]{Bo2})
is essential.
\begin{prop}
Every continuous homomorphism
between $p$-adic Lie groups
is analytic.
\end{prop}
\noindent
As a consequence, there is at most one
$p$-adic Lie group structure on a given
topological group.
Following the general custom, we call
a topological group a $p$-adic Lie group
if it admits a $p$-adic Lie
group structure.\\[2.3mm]
We record standard facts; for the proofs,
the reader is referred to
Proposition~7 in \S1, no.\,4;
Proposition~11 in \S1, no.\,6;
and Theorem~2 in \S8, no.\,2
in \cite[Chapter~3]{Bo2}.
\begin{prop}
Closed subgroups, finite direct products
and Hausdorff
quotient groups of $p$-adic
Lie groups are $p$-adic
Lie groups.\,\Punkt
\end{prop}
\noindent
It is also known that a topological
group which is an extension of a $p$-adic Lie group
by a $p$-adic Lie group is again a $p$-adic
Lie group. In fact,
the case of compact $p$-adic Lie groups
is known in the theory of analytic pro-$p$-groups
(one can combine \cite[Proposition~8.1.1\,(b)]{Wls}
with Proposition~1.11\,(ii)
and Corollary~8.33 from \cite{Dix}).
The general case then is a well-known consequence
(see, e.g., \cite[Lemma~9.1\,(a)]{SIM}).\\[2.1mm]
The following fact is essential for us
(cf.\ Step~1 in the proof of Theorem~3.5 in
\cite{Wan}).
\begin{prop}\label{veryess}
Every $p$-adic Lie group~$G$
has an open subgroup~$U$ which
satisfies the ascending chain condition
on closed subgroups.
\end{prop}
\begin{proof}
We show that every BCH-subgroup
$U$ has the desired property.
It suffices to discuss $U\sub L(G)$
with the BCH-multiplication.
It is known that the Lie algebra $L(H)$ of a closed
subgroup $H\leq U$ can be identified
with the set of all $x\in L(U)=L(G)$ such that
\[
\exp_U(W x)\,\sub\, H
\]
for some $0$-neighbourhood $W\sub \Q_p$
(see Corollary~1\,(ii) in \cite[Chapter~3, \S4, no.\,4]{Bo2}).
Since $\exp_U=\id_U$ here (see~(\ref{simplexp}))
and $\Z_p x = \wb{\langle x\rangle}\sub H$
for each $x\in H$ (by (\ref{newdetai2})),
we deduce that
\[
L(H)\;=\; \Spann_{\Q_p}(H)
\]
is the linear span of~$H$ in the current situation.
Now consider an ascending series
$H_1\leq H_2\leq\cdots$ of closed
subgroups. We may assume that
each $H_n$ has the same dimension
and thus $\ch:=L(H_1)=L(H_2)=\cdots$
for each~$n$.
Then $H:=\ch\cap U$ is a compact group
in which $H_1$ is open, whence
$[H:H_1]$ is finite and the series
has to stabilize.
\end{proof}
\noindent
We record an important consequence:
\begin{prop}\label{unoclos}
If $G$ is a $p$-adic Lie group
and $H_1\sub H_2\sub\cdots$
an
ascending sequence of closed
subgroups of~$G$, then
also
$H:=\bigcup_{n\in \N}H_n$ is closed in~$G$.
\end{prop}
\begin{proof}
Let $V\sub G$ be an open subgroup
satisfying an ascending chain condition
on closed subgroups.
Then
\[
V\cap H_1\; \sub\;  V\cap H_2\; \sub\; \cdots
\]
is an ascending sequence
of closed subgroups of~$V$ and thus
stabilizes, say at~$V\cap H_m$.
Then
\[
V\cap H\; =\; V\cap H_m
\]
is closed in~$G$. Being locally closed,
the subgroup~$H$ is closed (compare\linebreak
\cite[Theorem (5.9)]{HaR}).
\end{proof}
\noindent
Two important applications are now described.
The second one (Corollary~\ref{corclo})
is part of \cite[Theorem~3.5\,(ii)]{Wan}.
The first application might also be deduced
from the second using \cite[Theorem~3.32]{BaW}.
\begin{numba}\label{rcasca}
We recall from~\cite{Wil} and \cite{Wi2}:
If $G$ is a totally disconnected, locally compact
group and~$\alpha$ an automorphism of~$G$,
then a compact open subgroup $V\sub G$
is called \emph{tidy for $\alpha$}
if it has the following properties:
\begin{description}
\item[TA]
$V=V_+V_-$,
where $V_+:=\bigcap_{n\in\N_0}\alpha^n(V)$
and $V_-:=\bigcap_{n\in\N_0}\alpha^{-n}(V)$; and
\item[TB]
The ascending union
$V_{++}:=\bigcup_{n\in \N_0}\alpha^n(V_+)$
is closed in~$G$.
\end{description}
If $V$ satisfies {\bf TA}, it is also called \emph{tidy above}.
If $V$ satisfies {\bf TB}, it is also called \emph{tidy below}.
\end{numba}
%
%
\begin{cor}\label{autotb}
Let $G$ be a $p$-adic Lie group,
$\alpha\colon G\to G$ be an
automorphism
and $V\sub G$ be a compact open subgroup.
Then~$V$ is tidy below.
\end{cor}
\begin{proof}
The subgroup $\bigcup_{n\in \N_0}\alpha^n(V_+)$
is an ascending union of closed
subgroups of~$G$ and hence closed, by Proposition~\ref{unoclos}.
\end{proof}
\noindent
The second application
of Proposition~\ref{unoclos}
concerns contraction groups.
\begin{defn}\label{maindefcon}
Given a topological group~$G$
and automorphism $\alpha\colon G\to G$,
we define the \emph{contraction group}\footnote{Some authors
may prefer to call~$\alpha$ the contraction \emph{subgroup} of~$\alpha$.
Another recent notation for $U_\alpha$ is $\mbox{con}(\alpha)$.}
\emph{of $\alpha$} via
\begin{equation}\label{defU}
U_\alpha\, :=\, \{x\in G\colon
\mbox{$\alpha^n(x)\to 1\,$ as $\, n\to\infty$}\}\,.
\end{equation}
\end{defn}
\begin{cor}\label{corclo}
Let $G$ be a $p$-adic Lie group.
Then the contraction group $U_\alpha$ is closed in~$G$,
for each automorphism
$\alpha\colon G\to G$.
\end{cor}
\begin{proof}
Let $V_1\supseteq V_2\supseteq\cdots$ be a sequence
of compact open subgroups of~$G$ which
form a basis of identity neighbourhoods
(cf.\ Lemma~\ref{laballs}).
Then an element $x\in G$ belongs to $U_\alpha$
if and only if
\[
(\forall n\in \N)\, (\exists m\in \N)\,(\forall k\geq m)\;\;
\alpha^k(x)\in V_n\,.
\]
Since
$\alpha^k(x)\in V_n$ if and only if $x\in \alpha^{-k}(V_n)$,
we deduce that
\[
U_\alpha\; =\; \bigcap_{n\in\N}\bigcup_{m\in \N}\bigcap_{k\geq m}
\alpha^{-k}(V_n)\,.
\]
Note that $W_n:=\bigcup_{m\in \N}\bigcap_{k\geq m}
\alpha^{-k}(V_n)$ is an ascending union of closed subgroups
of~$G$ and hence closed, by Proposition~\ref{unoclos}.
Consequently, $U_\alpha=\bigcap_{n\in \N}\, W_n$
is closed.
\end{proof}
\begin{rem}
We mention that
contraction groups of the form~$U_\alpha$
arise in many contexts:
In representation theory
in connection with the Mautner phenomenon
(see \cite[Chapter~II, Lemma~3.2]{Mar}
and (for the $p$-adic case) \cite{Wan});
in probability theory
on groups (see \cite{HaS}, \cite{Sie}, \cite{Si2}
and (for the $p$-adic case)~\cite{DaS});
and in the structure theory
of totally disconnected, locally compact groups~\cite{BaW}.
\end{rem}
\section{Lazard's characterization of {\boldmath$p$}-adic
groups}\label{sec3}
Lazard~\cite{Laz} obtained various
characterizations of $p$-adic Lie groups
within the class of locally compact groups,
and many more have been found since
in the theory of analytic pro-$p$-groups
(see \cite[pp.\,97--98]{Dix}).
These characterizations (and the underlying theory
of analytic pro-$p$-groups)
are of great value
for the structure theory of totally
disconnected groups.
For example, they explain the occurrence
of $p$-adic Lie groups in
the areas of $\Z^n$-actions
and contractive automorphisms
mentioned in the introduction.
We recall one of Lazard's characterizations
in a form recorded in \cite[p.\,157]{Ser}:
\begin{thm}
A topological group~$G$ is a $p$-adic
Lie group if and only if it has an open
subgroup~$U$ with the following
properties:
\begin{itemize}
\item[\rm(a)]
$U$ is a pro-$p$-group;
\item[\rm(b)]
$U$ is topologically finitely generated,
i.e., $U=\wb{\langle F\rangle}$
for a finite subset $F\sub U$;
\item[\rm(c)]
$[U,U]\, :=\, \langle xyx^{-1}y^{-1}\colon x,y\in U\rangle
\, \sub \, \{x^{p^2}\colon x\in U\}$.
\end{itemize}
\end{thm}
\noindent
Although the proof of the sufficiency of
conditions (a)--(c) is non-trivial,
their\linebreak
necessity is clear.
In fact, (a) is immediate from
(\ref{goodcf}) and
Remark~\ref{smidx},
while~(c) can easily be proved
using the second order Taylor
expansion of the commutator map
and the inverse function theorem (applied to the
map $x\mto x^{p^2}$).
To obtain~(b),
one picks an exponential map
$\exp_G\colon V\to G$ as well as
a basis $x_1,\ldots, x_d\in V$ of $L(G)$,
and notes that the analytic map
\[
\phi\colon (\Z_p)^d\to G\,, \quad (t_1,\ldots, t_d)
\mto \exp_G(t_1x_1)\cdot \ldots\cdot \exp_G(t_dx_d)
\]
has an invertible differential
at the origin (the map $(t_1,\ldots, t_d)\mto
\sum_{j=1}^d t_jx_j$),
and hence restricts to
a diffeomorphism from $p^n\Z_p^d$ (for some $n\in \N_0$) 
onto an open identity neighbourhood~$U$,
which can be chosen as subgroup of~$G$
(by Lemma~\ref{laballs}).
Replacing the elements $x_j$ with $p^nx_j$,
we can always achieve that
$\phi$
is a diffeomorphism from $\Z_p^d$
onto a compact open subgroup~$U$ of~$G$
(occasionally, one speaks of
``coordinates of the second kind'' in this situation).
Then
\[
\wb{\langle \exp_G(x_1),\ldots,\exp_G(x_d)\rangle}\;=\; U\,,
\]
establishing~(b).
\section{Iteration of linear automorphisms}\label{secitlin}
A good understanding of linear automorphisms
of vector spaces over
local fields is essential
for an understanding of automorphisms
of general Lie groups.
\subsection*{Decomposition of {\boldmath$E$}
and adapted norms}\label{adnm}
For our first consideration,
we let $(\K,|.|)$ be a complete ultrametric
field,
$E$ be a finite-dimensional $\K$-vector
space, and $\alpha\colon E\to E$
be a linear automorphism.
We let $\wb{\K}$ be an algebraic closure of~$\K$,
and use the same symbol,
$|.|$, for the unique extension
of the given absolute value to~$\wb{\K}$
(see \cite[Theorem~16.1]{Sch}).
We let $R(\alpha)$ be the set
of all absolute values $|\lambda|$,
where $\lambda\in \wb{\K}$
is an eigenvalue
of the automorphism $\alpha_{\wb{\K}}:=\alpha\otimes \id_{\wb{\K}}$
of $E_{\wb{\K}}:=E\tensor_\K \wb{\K}$
obtained by extension of scalars.
For each $\lambda\in \wb{\K}$,
we let
\[
(E_{\wb{\K}})_{(\lambda)}\, :=\, \{ x\in E_{\wb{\K}}\!:
(\alpha_{\wb{\K}}-\lambda)^dx=0\}
\]
be the generalized eigenspace of $\alpha_{\wb{\K}}$
in $E_{\wb{\K}}$ corresponding to~$\lambda$
(where $d$ is the dimension of the $\K$-vector space~$E$).
Given $\rho\in R(\alpha)$,
we define
\begin{equation}\label{dfspacerho}
(E_{\wb{\K}})_\rho\; :=\;
\bigoplus_{|\lambda|=\rho} (E_{\wb{\K}})_{(\lambda)}\,\sub E_{\wb{\K}}
\, ,\vspace{-.7mm}
\end{equation}
where the sum is taken over all
$\lambda\in \wb{\K}$
such that $|\lambda|=\rho$. As usual, we identify $E$
with $E \otimes 1\sub E_{\wb{\K}}$.\\[2.1mm]
The following fact (cf.\ (1.0) on p.\,81 in \cite[Chapter~II]{Mar})
is essential:\footnote{In \cite[p.\,81]{Mar},
$\K$ is a local field,
but the proof works also for complete ultrametric fields.}
\begin{la}\label{lemschlem}
For each $\rho\in R(\alpha)$,
the vector
subspace $(E_{\wb{\K}})_\rho$ of $E_{\wb{\K}}$ is defined
over~$\K$, i.e.,
$(E_{\wb{\K}})_\rho= (E_\rho)_{\wb{\K}}$,
where $E_\rho:=(E_{\wb{\K}})_\rho\cap E$.
Thus
\begin{equation}\label{isdsum}
E\; =\; \bigoplus_{\rho\in R(\alpha)} E_\rho\,,
\end{equation}
and each $E_\rho$
is an $\alpha$-invariant vector subspace of~$E$.\,\Punkt
\end{la}
\noindent
It is useful for us that certain well-behaved norms
exist on~$E$ (cf.\ \cite[Lemma~3.3]{SCA}
and its proof for the $p$-adic case).
\begin{defn}
A norm $\|.\|$
on $E$ is \emph{adapted to $\alpha$} if the following holds:
\begin{itemize}
\item[{\bf A1}]
$\|.\|$ is ultrametric;
\item[{\bf A2}]
$\big\|\sum_{\rho\in R(\alpha)} x_\rho\big\|=\max\{\|x_\rho\|
\colon \rho\in R(\alpha)\}$\vspace{.5mm}
for each $(x_\rho)_{\rho\in R(\alpha)}
\in \prod_{\rho\in R(\alpha)} E_\rho$; and
\item[{\bf A3}]
$\|\alpha (x)\|=\rho\|x\|$ for each $\rho\in R(\alpha)$
and $x\in E_\rho$.
\end{itemize}
\end{defn}
\begin{prop}\label{propadapt}
Let $E$ be a finite-dimensional vector space over
a complete ultrametric field $(\K,|.|)$ and
$\alpha\colon E\to E$ be a linear automorphism.
Then $E$ admits a norm $\|.\|$ adapted to~$\alpha$.
\end{prop}
\noindent
This follows from the next lemma.
\begin{la}\label{normrho}
For each $\rho\in R(\alpha)$,
there exists an ultrametric norm
$\|.\|_\rho$ on $E_\rho$ such that $\|\alpha(x)\|_\rho=\rho\|x\|_\rho$
for each $x\in E_\rho$.
\end{la}
\begin{proof}
Assume $\rho\geq 1$ first.
We choose a $\wb{\K}$-basis $w_1,\ldots, w_m$ of $(E_{\wb{\K}})_\rho=:V$
such that the matrix $A_\rho$ of $\alpha_{\wb{\K}}|_V^V\colon V\to V$
with
respect to this basis has Jordan canonical
form. We let $\|.\|$ be the maximum norm
on $V$ with respect to this basis.\\[2.1mm]
If $\rho=1$, then $A_\rho\in \GL_m(\wb{\bO})$, where
$\wb{\bO}=\{t\in \wb{\K}\colon |t|\leq 1\}$
is the valuation ring of $\wb{\K}$, and hence
$\alpha_{\wb{\K}}|_V^V$
is an isometry with respect to~$\|.\|$
(by Proposition~\ref{propisos}).
Thus $\|\alpha_{\wb{\K}}(x)\|=\|x\|=\rho\|x\|$ for all $x\in V$.\\[2.1mm]
If $\rho>1$, we note that
for each $k\in\{1,\ldots, m\}$, there exists an
eigenvalue $\mu_k$
such that $w_k\in (E_{\wb{\K}})_{(\mu_k)}$.
Then $\alpha_{\wb{\K}}(w_k)=\mu_k w_k+(\alpha_{\wb{\K}}(w_k)-\mu_k w_k)$
for each~$k$,
with $\|\mu_k w_k\|=\rho >1\geq\|\alpha_{\wb{\K}}(w_k)-\mu_k w_k\|$.
As a consequence,
$\|\alpha_{\wb{\K}}(x)\|=\rho \|x\|$ for each $x\in V $,
using~(\ref{winner}).\\[2.1mm]
In either of the preceding cases, we define $\|.\|_\rho$
as the restriction of $\|.\|$ to $E_\rho$.
To complete the proof, assume $\rho<1$ now.
Then $E_\rho =E_{\rho^{-1}} (\alpha^{-1})$,
where $\rho^{-1}>1$.
Thus, by what has already been shown, there exists
an ultrametric norm $\|.\|_\rho$ on $E_\rho$
such that $\|\alpha^{-1}(x)\|_\rho=\rho^{-1}\|x\|_\rho$ for each $x\in E_\rho$.
Then $\|\alpha(x)\|_\rho=\rho\|x\|_\rho$ for each $x\in E_\rho$, as required.
\end{proof}
\noindent
{\bf Proof of Proposition~\ref{propadapt}.}
For each $\rho\in R(\alpha)$, we choose a norm $\|.\|_\rho$
on $E_\rho$ as described in Lemma~\ref{normrho}. Then
\[
\Big\| \sum_{\rho\in R(\alpha)} x_\rho\Big\| \; :=
\; \max\, \big\{ \,\|x_\rho\|_\rho\colon \rho\in R(\alpha)\,\big\}\quad
\mbox{for $\,(x_\rho)_{\rho\in R(\alpha)} \in \prod_{\rho\in R(\alpha)}
E_\rho$}
\]
defines a norm $\|.\|\colon E\to[0,\infty[$
which, by construction, is adapted to~$\alpha$.\,\Punkt
\subsection*{Contraction group
and Levi factor for a linear automorphism}\label{cotlevlin}
For $\alpha$ a linear automorphism of a finite-dimensional
vector space~$E$ over a local field~$\K$,
we now determine the contraction group
(as introduced in Definition~\ref{maindefcon})
and the associated Levi factor,
defined as follows:
\begin{defn}\label{basecontr}
Let $G$ be a topological group
and $\alpha\colon G\to G$ be an automorphism.
Following \cite{BaW}, we define
the \emph{Levi factor} $M_\alpha$
as the set of all $x\in G$ such that
the two-sided orbit $\alpha^\Z(x)$
is relatively compact in~$G$.
\end{defn}
\noindent
It is clear that both $U_\alpha$ and $M_\alpha$
are subgroups. If $G$ is locally compact
and totally disconnected,
then $M_\alpha$ is closed,
as can be shown using
tidy subgroups as a tool
(see \cite[p.\,224]{BaW}).
\begin{prop}\label{linearcase}
Let $\alpha$ be a linear automorphism
of a finite-dimensional
$\K$-vector space~$E$.
Then
\[
U_\alpha\, = \, \bigoplus_{\stackrel{{\scriptstyle\rho\in R(\alpha)}}{\rho<1}}
E_\rho\, ,\quad
M_\alpha \,= \, E_1
\quad\mbox{ and }\quad
U_{\alpha^{-1}} \, = \,
\bigoplus_{\stackrel{{\scriptstyle\rho\in R(\alpha)}}{\rho>1}} E_\rho \, .
\]
Furthermore, $E=U_\alpha \oplus M_\alpha\oplus U_{\alpha^{-1}}$
as a $\, \K$-vector space.
\end{prop}
\begin{proof}
Using an adapted norm on~$E$,
the characterizations of $U_\alpha$,
$M_\alpha$
and $U_{\alpha^{-1}}$ are clear.
That $E$ is the indicated direct sum
follows from (\ref{isdsum}).
\end{proof}
\subsection*{Tidy subgroups and the scale
for a linear automorphism}\label{secpre5}
In the preceding situation, define
\begin{equation}\label{Eplus}
E_+\; :=\; \bigoplus_{\rho\geq 1}E_\rho\quad
\mbox{and}\quad
E_-\; :=\; \bigoplus_{\rho\leq 1}E_\rho\, .
\end{equation}
\begin{prop}\label{proptdlin}
For each $r>0$ and norm $\|.\|$ on~$E$
adapted to~$\alpha$, the ball
$B_r:=\{x\in E\colon \|x\|<r\}$
is tidy for~$\alpha$,
with
\begin{equation}\label{givestidy}
(B_r)_\pm\; =\; B_r\cap E_\pm\, .
\end{equation}
The scale of $\alpha$ is given by
\begin{equation}\label{stasta}
s_E(\alpha)\;\,
=\;\, \Delta_{E_+}(\alpha|_{E_+}^{E_+})\;\,
=\;\, |\hspace*{-.4mm}\dt(\alpha|_{E_+}^{E_+})\hspace*{.2mm}| \, \;=
\prod_{\stackrel{{\scriptstyle j
\in \{1,\ldots, d\}}}{{\scriptstyle|\lambda_j|\geq 1}}}|\lambda_j|\, ,
\end{equation}
where $\lambda_1,\ldots, \lambda_d$
are the eigenvalues of~$\alpha_{\wb{\K}}$
$($occurring with multiplicities$)$.
\end{prop}
\begin{proof}
Let $x\in E$.
It is clear from the definition
of an adapted norm that $\alpha^{-n}(x)\in B_r$
for all $n\in \N_0$ (i.e., $x\in (B_r)_+$)
if and only if $x\in \bigoplus_{\rho\geq 1}E_\rho=E_+$.
Thus $(B_r)_+$ is given by (\ref{givestidy}),
and $(B_r)_-$ can be discussed
analogously.
It follows from property {\bf A2} in
the definition of an adapted norm that
\begin{equation}\label{dubstr}
B_r\; =\; \bigoplus_{\rho\in R(\alpha)} (B_r\cap E_\rho)\, ,
\end{equation}
and thus
$B_r=(B_r)_++(B_r)_-$.
Hence $B_r$ is tidy above.
Since
\[
(B_r)_{++}\; =\; \bigcup_{n\in \N_0}\alpha^n((B_r)_+)
\;=\; (B_r\cap E_1)\oplus \bigoplus_{\rho>1}E_\rho 
\]
is closed, $B_r$ is also tidy below
and thus $B_r$ is tidy.\\[2.1mm]
Let $\lambda$ be a Haar measure on~$E_+$.
Since $K:=(B_r)_+$ is open in $E_+$,
compact, and $K\sub \alpha(K)$,
we obtain
\[
\Delta_{E_+}\big(\alpha|_{E_+}^{E_+}\big)\,=\,
\frac{\lambda(\alpha(K))}{\lambda(K)}\,=\,
\frac{[\alpha(K): K]\,\lambda(K)}{\lambda(K)}\,=\,
[\alpha(K): K]\,=\, s_E(\alpha)\,,
\]
using the translation invariance of Haar measure for
the second equality. Here
\[
\Delta_{E_+}\big(\alpha|_{E_+}^{E_+}\big)
\; =\; \big|\hspace*{-.4mm}\dt\big(\alpha|_{E_+}^{E_+}\big)\hspace*{.1mm}\big|
\; = \, {\displaystyle \prod_{\stackrel{{\scriptstyle j\in
\{1,\ldots, d\}}}{|\lambda_j|\geq 1}}|\lambda_j|}\,,
\]
by
(\ref{modvia}) in
Lemma~\ref{laindx}.
\end{proof}
\section{Scale and tidy subgroups for automorphisms
of {\boldmath $p$}-adic Lie groups}\label{sec5}
In this section, we determine tidy subgroups and calculate the scale
for\linebreak
automorphisms of a $p$-adic Lie group.
\begin{numba}\label{nwnbr}
Applying (\ref{locnat})
to an automorphism $\alpha$ of a $p$-adic
Lie group~$G$,
we find an exponential function
$\exp_G\colon V\to G$
which is a diffeomorphism onto its
image, and an open $0$-neighbourhood
$W\sub V$ such that $L(\alpha).W\sub V$
and
\begin{equation}\label{locln}
\exp_G\circ L(\alpha)|_W\; =\; \alpha\circ \exp_G|_W\,.
\end{equation}
Thus $\alpha$ is locally linear
in a suitable (logarithmic) chart,
which simplifies an understanding of the dynamics
of~$\alpha$.
Many aspects can be reduced to
the dynamics of the linear automorphism
$\beta:=L(\alpha)$ of $L(G)$,
as discussed in Section~\ref{secitlin}.
After shrinking~$W$, we may assume that also
$L(\alpha^{-1}).W\sub V$
and
\begin{equation}\label{invneweq}
\exp_G \circ \, L(\alpha^{-1})|_W\; =\;
\alpha^{-1}\circ \exp_G|_W\,.
\end{equation}
\end{numba}
\noindent
To construct subgroups tidy for~$\alpha$,
we pick a norm $\|.\|$ on $E:=L(G)$ adapted
to $\beta:=L(\alpha)$,
as in Proposition~\ref{propadapt}.
There is $r>0$ such that $B_r:=B_r^E(0)\sub W$
and
\[
V_t\; :=\; \exp_G(B_t)
\]
is a compact open subgroup
of~$G$, for each $t\in \;]0,r]$
(see Lemma~\ref{laballs}).
Abbreviate
$Q_t:=B_t\cap \bigoplus_{\rho<1}\, E_\rho$.
Since $B_t=(B_t)_+ \oplus Q_t$
by (\ref{dubstr})
(applied to the linear automorphism
$\beta$),
the ultrametric inverse function theorem
(Proposition~\ref{invfct}) and (\ref{Tay01})
imply that
$V_t = \exp_G((B_t)_+)\exp_G(Q_t)$ and thus
\begin{equation}\label{gta}
V_t\;=\; \exp_G((B_t)_+)\exp_G((B_t)_-)
\end{equation}
for all $t\in \;]0,r]$,
after shrinking $r$ if necessary.
Then we have (as first recorded in \cite[Theorem~3.4\,(c)]{MIX}):
\begin{thm}\label{scainft}
The subgroup $V_t$ of~$G$
is tidy for~$\alpha$, for each $t\in \;]0,r]$.
Moreover,
\[
s_G(\alpha)\;=\; s_{L(G)}(L(\alpha))\,,
\]
where $s_{L(G)}(L(\alpha))$ can be calculated as in
Proposition~{\rm\ref{proptdlin}}.
\end{thm}
\begin{proof}
Since $\beta((B_t)_-)\sub (B_t)_-$
and $\beta^{-1}((B_t)_+)\sub (B_t)_+$,
it follows with (\ref{locln}),
(\ref{invneweq}) and induction that
\[
\alpha^{\mp n}(\exp_G((B_t)_{\pm}))\; = \;
\exp_G (\beta^{\mp n}((B_t)_\pm))\; \sub \; \exp_G((B_t)_\pm)
\]
for each $n\in \N_0$ and hence
\[
\exp_G((B_t)_\pm)\; \sub \; (V_t)_\pm\,.
\]
Therefore $V_t=(V_t)_+(V_t)_-$
(using (\ref{gta}))
and thus~$V_t$ is tidy above.
By Corollary~\ref{autotb}, $V_t$ is also tidy below
and thus~$V_t$ is tidy for~$\alpha$.\\[2.1mm]
Closer inspection shows that
$L((V_t)_\pm)=L(G)_\pm$ (see \cite[Theorem~3.5]{SCA}
or proof of Theorem~\ref{thmposca} below),
where $L(G)_\pm$
is defined as in (\ref{Eplus}),
using~$\beta$.
Since the module of a Lie group
automorphism can be calculated
on the Lie algebra level (using
Proposition~55\,(ii)
in \cite[Chapter 3, \S3, no.\,16]{Bo2}),
it follows that
$s_G(\alpha)=\Delta_{(V_t)_{++}}(\alpha|_{(V_t)_{++}})
=\Delta_{L(G)_+}(L(\alpha)|_{L(G)_+})
=s_{L(G)}(L(\alpha))$.
\end{proof}
\section{{\boldmath $p$-adic} contraction groups}\label{sec6}
Exploiting the local linearity
of~$\alpha$ (see (\ref{locln}))
and the closedness
of $p$-adic contraction groups
(see Corollary~\ref{corclo}),
it is possible to reduce the discussion
of $p$-adic contraction groups
to contraction groups
of linear automorphisms,
as in Proposition~\ref{linearcase}.
We record the results, due to Wang~\cite[Theorem~3.5]{Wan}.
\begin{thm}\label{Wa1}
Let $G$ be a $p$-adic Lie group
and $\alpha\colon G\to G$ be an
automorphism.
Abbreviate $\beta:=L(\alpha)$.
Then
$U_\alpha$, $M_\alpha$
and $U_{\alpha^{-1}}$
are closed subgroups
of~$G$ with Lie algebras
\[
L(U_\alpha)=U_\beta\,,\quad
L(M_\alpha)=M_\beta
\quad\mbox{and}\quad
L(U_{\alpha^{-1}})=U_{\beta^{-1}}\, ,
\]
respectively.
Furthermore,
$U_\alpha M_\alpha U_{\alpha^{-1}}$
is an open identity neighbourhood in~$G$,
and the product map
\[
U_\alpha \times M_\alpha \times U_{\alpha^{-1}}
\to U_\alpha M_\alpha U_{\alpha^{-1}}\,,
\quad
(x,y,z)\mto xyz
\]
is an analytic diffeomorphism.\,\Punkt
\end{thm}
\noindent
In Part\,(ii) of his Theorem~3.5, Wang also obtained essential
information concerning the groups $U_\alpha$:
\begin{thm}\label{Wa2}
Let~$G$ be a $p$-adic
Lie group admitting a contractive
automorphism~$\alpha$.
Then $G$ is a unipotent linear algebraic
group defined over~$\Q_p$,
and hence nilpotent.\,\Punkt
\end{thm}
\noindent
We remark that Theorem~\ref{Wa1}
becomes false in general
for Lie groups over local fields
of positive characteristic (as illustrated in~\ref{ncgp}).
However, we shall see later that its conclusions
remain
valid if closedness
of~$U_\alpha$ is made an extra
hypothesis.
Assuming closedness of~$U_\alpha$,
a certain
analogue
of Theorem~\ref{Wa1}
can even be obtained in a purely topological
setting:
\begin{prop}\label{propmetr}
Let $G$ be a totally disconnected,
locally compact group
and $\alpha\colon G\to G$ be
an automorphism.
If~$U_\alpha$ is closed,
then $U_\alpha M_\alpha U_{\alpha^{-1}}$
is an open identity neighbourhood in~$G$
and the product map
\begin{equation}\label{prodump}
\pi\colon U_\alpha \times M_\alpha \times U_{\alpha^{-1}}
\to U_\alpha M_\alpha U_{\alpha^{-1}}\,,
\quad
(x,y,z)\mto xyz
\end{equation}
is a homeomorphism.
\end{prop}
\begin{proof}
If $U_\alpha$ is closed, then every identity neighbourhood
in $G$ contains a compact open subgroup tidy for~$\alpha$
(see \cite{BaW} if $G$ is metrizable;
the results from \cite{Jaw}
can be used to remove the metrizability condition).
Thus $\alpha$ is ``tidy'' in the terminology of
\cite{TID}. Therefore, the proposition
is covered by part (f) of the theorem
in \cite{TID}
(the proof of which heavily uses results from \cite{BaW}).
\end{proof}
\noindent
We conclude the section with an example for the calculations
in Sections~\ref{secpre5}--\ref{sec6}.
\begin{example}
We consider the $p$-adic Lie group $G:=\GL_2(\Q_p)$
and its inner automorphism $\alpha\colon A\mto gAg^{-1}$
given by
\[
g:=\left(\begin{array}{cc}
1 & 0 \\
0 & p
\end{array}\right).
\]
For
\[
A=\left(\begin{array}{cc}
a & b \\
c & d
\end{array}\right)\in G,
\]
we have
\[
g^nAg^{-n}=
\left(\begin{array}{cc}
a & p^{-n}b \\
p^nc & d
\end{array}\right)
\]
for all $n\in \Z$, entailing that $M_\alpha\sub G$ is the subgroup
of all invertible diagonal matrices,
\[
U_\alpha=\left\{
\left(\begin{array}{cc}
1 & 0 \\
c & 1
\end{array}\right)\colon c\in \Q_p\right\}
\quad\mbox{and}\quad
U_{\alpha^{-1}}=\left\{
\left(\begin{array}{cc}
1 & b \\
0 & 1
\end{array}\right)\colon b\in \Q_p\right\}.
\]
Since $G$ is an open subset of the space $M_2(\Q_p)$
of $p$-adic $(2\times 2)$-matrices, we can identify
the tangent space $L(G)$ at the identity element with
$M_2(\Q_p)$ (as usual). Then $L(\alpha)$ corresponds
to the linear automorphism
\[
\beta\colon M_2(\Q_p)\to M_2(\Q_p),\quad
\left(\begin{array}{cc}
a & b \\
c & d
\end{array}\right)\mto
\left(\begin{array}{cc}
a & p^{-1}b \\
pc & d
\end{array}\right).
\]
Note
that the matrix units $E_{j,k}$ with only one
non-zero-entry, $1$,
in the $j$th row and $k$th column, form a basis of $\beta$-eigenvectors
for $M_2(\Q_p)$
with eigenvalues $p$ and $p^{-1}$ (of multiplicity one)
and $1$ as an eigenvalue of multiplicity~$2$.
Thus
\[
s_G(\alpha)=s_{L(G)}(\beta)=|1|\cdot |1|\cdot |p^{-1}|=p,
\]
by Theorem~\ref{scainft} and Proposition~\ref{proptdlin}.
It is clear from the preceding decomposition in eigenspaces
that the maximum-norm on $M_2(\Q_p)$,
$\|A\|_\infty:=\max\{|a|,|b|,|c|,|d|\}$ for $A$
as above, is adapted to~$\beta$.
Therefore
\[
B_r:=\{A\in M_2(\Q_p)\colon \|A\|_\infty<r\}
\]
is tidy for~$\beta$ for each $r>0$
(see Proposition~\ref{proptdlin}).
As the matrix exponential function
\[
\exp\colon U\to \GL_2(\Q_p)
\]
(defined on some open $0$-neighbourhood $U\sub M_2(\Q_p)$)
is an exponential function for~$G$ (see \cite{Bo2}),
we deduce with Theorem~\ref{scainft} that
\[
V_r:=\exp(B_r)
\]
is a tidy subgroup for~$\alpha$ for all sufficiently small $r>0$.
We mention that $V_r=\one+B_r$ coincides
with the ball of radius~$r$ around the identity matrix
for small~$r$, by Proposition~\ref{invfct},
using that the derivative $\exp'(0)=\id_{M_2(\Q_p)}$
is an isometry of $(M_2(\Q_p),\|.\|_\infty)$.
\end{example}
\section{Pathologies in positive characteristic}\label{secpatho}
Most of our discussion of automorphisms
of $p$-adic Lie groups becomes
false for Lie groups over local fields
of positive characteristic
(without extra
assumptions).
Suitable extra assumptions will
be described later.
For the moment, let
us have a look at some of the possible pathologies.\footnote{Compiled
from \cite{SPO}.}
\subsection{Non-closed contraction groups}\label{ncgp}
We describe an analytic automorphism
of an analytic Lie group over a local field
of positive characteristic,
with a non-closed contraction group.\\[2.1mm]
Given a finite field~$\F$
with prime order $|\F|=p$,
consider the compact group
$G:=\F^\Z$ and the right shift
\[
\alpha\colon G\to G\,, \quad \alpha(f)(n):=f(n-1)\,.
\]
Since the sets $\{f\in \F^\Z \colon f(-n)=f(-n+1)=\cdots=f(n-1)=f(n)=0\}$
(for $n\in \N$) form
a basis of $0$-neighbourhoods in~$G$,
it is easy to see
that the contraction group
$U_\alpha$ consists
exactly of the functions
with support bounded below,
i.e.,
\[
U_\alpha\;=\; \F^{(-\N)}\times \F^{\N_0}\,.
\]
This is a dense, non-closed subgroup
of~$G$. Also note that $G$ does not have\linebreak
arbitrarily small
subgroups tidy for~$\alpha$:
in fact, $G$ is the only
tidy subgroup.\\[2.1mm]
We now observe that $G$ can be considered
as a $2$-dimensional Lie group
over $\K:=\F(\!(X)\!)$:
The map
\[
G\to \F[\hspace*{-.17mm}[X]\hspace*{-0.16mm}]\times
\F[\hspace*{-.17mm}[X]\hspace*{-.16mm}]\,,\quad
f\mto \left(\sum_{n=1}^\infty f(-n)X^{n-1},\,
\sum_{n=0}^\infty f(n)X^n\right)
\]
is a global chart.
The automorphism of $\F[\hspace*{-.17mm}[X]\hspace*{-.16mm}]^2$
corresponding to~$\alpha$
coincides on the open $0$-neighbourhood
$X \F[\hspace*{-.17mm}[X]\hspace*{-.16mm}]\times
\F[\hspace*{-.17mm}[X]\hspace*{-.16mm}]$
with the linear map
\[
\beta\colon \K^2\to \K^2\,,\quad
\beta(v,w)=(X^{-1}v,Xw)\,.
\]
Hence $\alpha$ is an analytic
automorphism.
\subsection{An automorphism
whose scale cannot be\\
calculated on the Lie algebra level}\label{notlc}
We retain $G$, $\alpha$ and $\beta$ from the preceding example.\\[2mm]
Since $G$ is compact, we have $s_G(\alpha)=1$.
However, $L(\alpha)=\beta$
and therefore\linebreak
$s_{L(G)}(L(\alpha))=s_{\K(\!(X)\!)^2}(\beta)=|X^{-1}|=p$.
Thus
\[
s_G(\alpha)\;\not=\; s_{L(G)}(L(\alpha))\,,
\]
in stark contrast to the $p$-adic case,
where equality did always hold\,!
\subsection{An automorphism
which is not locally linear in any chart}\label{nln}
For $\F$, $\K$ and $p$ as before,
consider the $1$-dimensional
$\K$-analytic Lie group
$G:=\F[\hspace*{-.17mm}[X]\hspace*{-.16mm}]$.
Then
\[
\alpha\colon G\to G\,,\quad
z\mto z+ Xz^p
\]
is an analytic automorphism of~$G$
with $\alpha'(0)=\id$.
If $\alpha$ could be
locally
linearized, we could find a diffeomorphism
$\phi\colon U\to V$ between open $0$-neighbourhoods\linebreak
$U,V\sub G$ such that $\phi(0)=0$ and
\begin{equation}\label{givecontr}
\alpha\circ \phi|_W=\phi\circ \beta|_W
\end{equation}
for some linear map $\beta\colon \K\to\K$
and $0$-neighbourhood $W\sub G$.
Then $\phi'(0)=\alpha'(0)\circ
\phi'(0)=\phi'(0)\circ \beta$,
and hence $\beta=\id$.
Substituting this into~(\ref{givecontr}),
we deduce that $\alpha|_{\phi(W)}=\id_{\phi(W)}$,
which is a contradiction.
\section{Tools from non-linear analysis: invariant\\
\hspace*{.1mm}manifolds}\label{sec9}
As an automorphism $\alpha$ of a Lie group $G$
over a local field of positive characteristic
need not be locally linear
(see~\ref{nln}),
one has to resort to techniques
from non-linear analysis to
understand the iteration of such an automorphism.
The essential tools
are stable manifolds and centre
manifolds around a fixed point,
which are well-known tools
in the classical case of dynamical systems
on real manifolds.
The analogy is clear:
We are dealing with the (time-) discrete
dynamical system $(G,\alpha)$,
and are interested in its behaviour around
the fixed point~$1$.\\[2.1mm]
We shall use the following
terminology (cf.\ \cite{Bo2} and \cite{Ser}):
A subset $N$ of an $m$-dimensional
analytic manifold $M$ is called an $n$-dimensional
\emph{submanifold} of~$M$ if for each
$x\in N$, there is an analytic diffeomorphism
$\phi=(\phi_1,\ldots, \phi_m)
\colon U\to V$ from an open $x$-neighbourhood
$U\sub M$ onto an open subset $V\sub \K^m$
such that $\phi(U\cap N)=V\cap (\K^n\times \{0\})\sub
\K^n\times \K^{m-n}$. Then the maps $(\phi_1,\ldots,\phi_n)|_{U\cap N}$
form an atlas of charts for an analytic manifold structure on~$N$.
An \emph{immersed submanifold} of~$M$
is a subset $N\sub M$, together with an analytic manifold structure on~$N$
for which the inclusion map $\iota\colon N\to M$ is an analytic immersion
(i.e., an analytic map whose tangent maps
$T_x(\iota)\colon T_xN\to T_xM$ are injective for all $x\in N$).
If a subgroup $H$ of an analytic Lie group~$G$ is a submanifold,
then the inherited analytic manifold structure turns it into an analytic
Lie group. A subgroup $H$ of~$G$, together with an analytic
Lie group structure on~$H$, is called an \emph{immersed
Lie subgroup} of~$G$, if the latter turns the inclusion map
$H\to G$ into an immersion.\\[2.1mm]
To discuss stable manifolds in the ultrametric case,
we consider the following\linebreak
setting:
$(\K,|.|)$
is a complete
ultrametric field and
$M$ a finite-dimensional
analytic manifold over~$\K$, of dimension~$m$.
Also, $\alpha \colon M\to M$ is a given
analytic diffeomorphism
and $z \in M$ a fixed point of~$\alpha$.
We let $a\in \; ]0,1] \setminus R(T_z(\alpha))$,
using notation as in Section~{\rm\ref{secitlin}}.\\[2.1mm]
Now define $W^s_a(\alpha,z)$ as the set
of all $x\in M$
with the following property:\footnote{The letter ``$s$'' in $W^s_a(\alpha,z)$
stands for ``stable.''}
For some (and hence each)
chart $\phi\colon U\to V\sub \K^m$ of~$M$
around~$z$ such that $\phi(z)=0$,
and some (and hence each)
norm~$\|.\|$ on~$\K^m$,
there exists $n_0\in \N$
such that $\alpha^n(x)\in U$
for all integers $n\geq n_0$ and
\begin{equation}\label{asymp}
\lim_{n\to\infty} \frac{\|\phi(\alpha^n(x))\|}{a^n}\;=\; 0\,.
\end{equation}
It is clear from the definition
that $W^s_a(\alpha,z)$ is
an $\alpha$-stable subset of~$M$.
The following fact is obtained
if we combine \cite[Theorem 1.3]{STN} and
\cite[Theorem A]{STF} (along with its proof):
\begin{thm}\label{ultrasat}
For each $a\in \; ]0,1] \setminus R(T_z(\alpha))$,
the set $W^s_a(\alpha,z)$ is an immersed submanifold
of~$M$.
Its tangent space at~$z$ is
$\bigoplus_{\rho<a}T_z(M)_\rho$,
using notation as in {\rm Lemma~\ref{lemschlem}}
with
$\beta:=T_z(\alpha)\colon T_z(M)\to T_z(M)$
in place of~$\alpha$
and $E:=T_z(M)$.
\end{thm}
\noindent
As in the real case, we call
$W^s_a(\alpha,z)$ the \emph{$a$-stable
manifold} around~$z$.\\[2.1mm]
{\bf Some ideas of the proof.}
The main point is to construct a local
$a$-stable manifold, i.e.,
an $\alpha$-invariant
submanifold~$N$ of~$M$ tangent to
$\bigoplus_{\rho<a}T_z(M)_\rho$
such that $\alpha|_N\colon N\to N$
is an analytic map.\footnote{Then\vspace{-.6mm}
$W^s_a(\alpha,z)
=\bigcup_{n\in \N_0}\alpha^{-n}(N)$,
which is easily made an analytic manifold in such a way that
$N$ becomes
an open subset and $\alpha$ restricts
to an analytic diffeomorphism
of $W^s_a(\alpha,z)$.}
In the real case, M.\,C. Irwin~\cite{Ir1}
showed how
the construction
of a (local) $a$-stable manifold
can be reduced to the implicit function
theorem, applied to a Banach space
of sequences
(cf.\ also~\cite{Wel}).\footnote{The idea is
to construct
not the points $x$ of the manifold,
but their orbits $(\alpha^n(x))_{n\in \N_0}$.}
Since implicit function theorems
are available also
for analytic mappings between ultrametric Banach spaces
(see \cite{Bo1}),
it is possible to adapt Irwin's
method to the ultrametric case
(see \cite{STN}).\,\Punkt\\[2.1mm]
Other classical types of invariant manifolds
are also available in the ultrametric case.
In particular:
\begin{numba}\label{existscentr}
For $M$ and $\alpha$ as before,
there always exists a
so-called
\emph{centre manifold},
i.e., an immersed analytic submanifold~$C$ of~$M$
which is tangent to $M_\beta\sub T_z(M)$
at~$z$, $\alpha$-stable
(i.e., $\alpha(C)=C$),
and such that the restriction $f|_C\colon C\to C$ is
analytic \cite[Theorem~1.10]{STN}.
By Proposition~\ref{invfct},
after shrinking~$C$ we may assume that~$C$
is diffeomorphic to a ball
and hence compact.
\end{numba}
\noindent
The neutral element $1$ of a Lie group $G$
is a fixed point of each automorphism $\alpha$ of $G$,
but it need not be a hyperbolic fixed point,
i.e., it may very well happen that
$1\in R(T_1(\alpha))$. Nonetheless,
it is always possible to
make $U_\alpha=W^s_1(\alpha,1)$ a manifold
(see \cite[Theorem~D]{STF}):
\begin{prop}\label{nonhyper}
Let $G$ be an analytic Lie group
over a complete valued field,
$\alpha\colon G\to G$ be an analytic automorphism
and $\beta:=L(\alpha)$ the associated Lie algebra automorphism
of $L(G)$.
Then $U_\alpha$
can be made an immersed Lie subgroup
modelled on $U_\beta=W^s_1(\beta,0)\sub L(G)$,
such that $\alpha|_{U_\alpha}\colon U_\alpha\to U_\alpha$
is an analytic automorphism and contractive.\,\Punkt
\end{prop}
\begin{proof}
The idea is to show that
$U_\alpha=W^s_a(\alpha,1)$ for
$a\in \,]0,1[\, \setminus R(\beta)$
close to~$1$.
Details can be found in~\cite{STF}.
\end{proof}
\noindent
Although it always is an immersed Lie subgroup,
$U_\alpha$ need not be a Lie subgroup (as the example in
Section~\ref{secpatho} shows).
\section{The scale, tidy subgroups
and contraction\\
groups in positive characteristic}\label{sec10}
We begin with a discussion
of the contraction group
for a well-behaved\linebreak
automorphism~$\alpha$
of a Lie group~$G$ over a local field,
and the associated local decomposition
of $G$ adapted to~$\alpha$.
Using the tools from non-linear analysis
just described, one obtains
the following
analogue of Theorem~\ref{Wa1}
(see \cite{SPO}):
\begin{thm}\label{Wa3}
Let $G$ be an analytic Lie group over a local field
and $\alpha\colon G\to G$ be an analytic automorphism.
Abbreviate $\beta:=L(\alpha)$.
If $U_\alpha$ is closed, then
$U_\alpha$, $M_\alpha$
and $U_{\alpha^{-1}}$
are closed Lie subgroups
of~$G$ with Lie algebras
\[
L(U_\alpha)=U_\beta\,,\quad
L(M_\alpha)=M_\beta
\quad\mbox{and}\quad
L(U_{\alpha^{-1}})=U_{\beta^{-1}}\,,
\]
respectively.
Moreover,
$U_\alpha M_\alpha U_{\alpha^{-1}}$
is an open identity neighbourhood in~$G$,
and the product map
\begin{equation}\label{nowfinly}
\pi\colon U_\alpha \times M_\alpha \times U_{\alpha^{-1}}
\to U_\alpha M_\alpha U_{\alpha^{-1}}\,,
\quad
(x,y,z)\mto xyz
\end{equation}
is an analytic diffeomorphism.
\end{thm}
\noindent
{\bf Some ideas of the proof.}
By Proposition~\ref{nonhyper},
$U_\alpha$ and $U_{\alpha^{-1}}$
are
immersed Lie subgroups of~$G$
with Lie algebras $U_\beta$ and $U_{\beta^{-1}}$,
respectively.
By \ref{existscentr},
there exists a compact centre manifold~$C$
for $\alpha$ around~$1$,
modelled on $M_\beta$.
Since $\alpha^n(C)=C$
for each $n\in \Z$, we have $C\sub M_\alpha$.
Now $L(G)=U_\beta\oplus M_\beta\oplus M_{\beta^{-1}}$
(see Proposition~\ref{linearcase}).
By the inverse function theorem,
the product map
\[
U_\alpha\times C\times U_{\alpha^{-1}}\to G
\]
is a local diffeomorphism at $(1,1,1)$.
Using Proposition~\ref{propmetr},
we deduce that $C$ is open in $M_\alpha$,
and now a standard argument
(Proposition~18 in \cite[Chapter~3, \S1, no.\,9]{Bo2})
can be used to make $M_\alpha$ a Lie
group.
Then
$\pi$ from (\ref{prodump})
is a local diffeomorphism
at $(1,1,1)$, and
one deduces as in the proof
of Part~(f) in the theorem
in \cite{TID}
that~$\pi$ is an analytic diffeomorphism onto its image.\,\Punkt
\begin{rem}
Even if $U_\alpha$ is not closed,
its stable manifold structure
makes it an immersed Lie subgroup
of~$G$, and $\alpha|_{U_\alpha}$
is a contractive auto\-mor\-phism
of this Lie group (see Proposition~\ref{nonhyper}).
Therefore Section~\ref{sec11}
below provides structural
information on~$U_\alpha$
(no matter whether it is closed or not).
\end{rem}
\noindent
Next, we discuss tidy subgroups
and the scale for well-behaved
automorphisms of Lie groups
over local fields.\\[2.1mm]
Let $G$ be an analytic Lie group
over a local field,
and $\alpha$ be an analytic automorphism of~$G$.
Let $\|.\|$ be a norm on $L(G)$
adapted to $\beta:=L(\alpha)$
and $\phi\colon U\to V\sub L(G)$
a chart of~$G$ around~$1$ such that
$\phi(1)=0$ and $T_1\phi=\id_{L(G)}$.
We know from Lemma~\ref{laballs}
that $V_r:=\phi^{-1}(B_r(0))$
is a subgroup of $G$ if $r$ is small (where $B_r(0)\sub L(G)$). 
Then the following holds:
\begin{thm}\label{thmposca}
If $U_\alpha$ is closed,
then the subgroup $V_r$ of~$G$ $($as just defined$)$
is tidy for~$\alpha$, for each sufficiently small $r>0$.
Moreover,
\begin{equation}\label{hasanag}
s_G(\alpha)\;=\; s_{L(G)}(L(\alpha))\,,
\end{equation}
where $s_{L(G)}(L(\alpha))$ can be calculated as in
Proposition~{\rm\ref{proptdlin}}.
If $U_\alpha$ is not closed,
then $s_G(\alpha)\not=s_{L(G)}(L(\alpha))$;
more precisely, $s_G(\alpha)$
is a proper
divisor of $s_{L(G)}(L(\alpha))$.
\end{thm}
\noindent
{\bf Sketch of proof.} We explain why $V_r$ is tidy.
Combining Theorem~\ref{Wa3}
with the Ultrametric Inverse Function Theorem,
we obtain a diffeomorphic decomposition
\begin{equation}\label{diffdeco}
V_r\;=\; (V_r\cap U_\alpha) (V_r\cap M_\alpha)
(V_r\cap U_{\alpha^{-1}})
\end{equation}
for small~$r$.
Since $\beta$ takes $B_r(0)\cap U_\beta$
inside $B_{r\|\beta|_{U_\beta}\|_{\op}}(0)\cap U_\beta$
for each $r$ (where $\|\beta|_{U_\beta}\|_{\op}<1$),
the Ultrametric Inverse Function Theorem
implies that
\[
\bigcap_{n\in \N_0}\alpha^n(V_r\cap U_\alpha)=\{1\}
\]
for all small $r>0$.
Similarly, $\beta(B_r(0)\cap M_\beta)=B_r(0)\cap M_\beta$
and $\beta(B_r(0)\cap U_{\beta^{-1}})\supseteq B_r(0)\cap U_{\beta^{-1}}$
imply that
\[
\bigcap_{n\in \N_0}\alpha^n(V_r\cap M_\alpha)=V_r\cap M_\alpha\quad
\mbox{and}\quad
\bigcap_{n\in \N_0}\alpha^n(V_r\cap U_{\alpha^{-1}})=V_r\cap U_{\alpha^{-1}}
\]
for small~$r$.
Combining this information with~(\ref{diffdeco})
and (\ref{nowfinly}),
we see that
\[
(V_r)_+\; :=\; \bigcap_{n\in \N_0}\alpha^n(V_r)=
(V_r\cap M_\alpha)(V_r\cap  U_{\alpha^{-1}})\,.
\]
As a consequence,
$(V_r)_+$ has the Lie algebra $M_\beta\oplus U_{\beta^{-1}}$.
Likewise,
$(V_r)_-=(V_r\cap U_\alpha)(V_r\cap M_\alpha)$
and $L((V_r)_-)=U_\beta\oplus M_\beta$.
Hence $V_r=(V_r)_+(V_r)_-$,
by~(\ref{diffdeco}).
Thus $V_r$ is tidy above
(as in \ref{rcasca}).
Since $(V_r)_{++}= (V_r\cap M_\alpha)U_{\alpha^{-1}}$
as a consequence of
(\ref{nowfinly}) and the $\alpha$-stability of $V_r\cap M_\alpha$,
we see that $(V_r)_{++}$ is closed
(whence $V_r$ is also tidy below and hence tidy).
For the remaining assertions, see~\cite{SPO}.\,\Punkt
\begin{rem}
At a first glance, the discussion
of the scale and tidy subgroups
for automorphisms of Lie groups
are easier in the $p$-adic case
than in positive
characteristic because any
automorphism is locally linear
in the $p$-adic case.
But a more profound difference
is the automatic validity of
the tidiness property {\bf TB} for automorphisms
of $p$-adic Lie groups,
which becomes false in
positive characteristic.\footnote{Note that also the pathological
automorphism
$\alpha$ in Section~\ref{secpatho}
was locally linear.}
\end{rem}
\noindent
As observed in~\cite{SPO},
results from~\cite{BaW} imply
the following closedness criterion for~$U_\alpha$,
which
is frequently easy
to check.
\begin{prop}\label{gencdclo}
Let $G$ be a totally
disconnected, locally
compact group
and $\alpha$ be an automorphism
of~$G$. If
there exists an injective continuous homomorphism
$\phi\colon G\to H$ to
a totally disconnected, locally compact
group~$H$ and an automorphism
$\beta$ of $H$ such that $U_\beta$ is closed
and $\beta\circ \phi=\phi\circ\alpha$,
then also $U_\alpha$ is closed.
\end{prop}
\begin{proof}
It is known that $U_\beta\cap M_\beta=\{1\}$
if and only if $U_\beta$ is closed
(see \cite[Theorem~3.32]{BaW} if $H$ is metrizable;
the metrizability condition can be removed
using techniques from \cite{Jaw}),
which holds by hypothesis.
Since
$\phi(U_\alpha)\sub U_\beta$
and $\phi(M_\alpha)\sub M_\beta$,
the injectivity of~$\phi$ entails that
$U_\alpha\cap M_\alpha=\{1\}$.
Hence $U_\alpha$ is closed,
using \cite[Theorem~3.32]{BaW} again.
\end{proof}
\noindent
Since $U_\beta\cap M_\beta=\{1\}$
holds for each inner automorphism
of $\GL_n(\K)$ (compare\linebreak
Proposition~\ref{linearcase}),
the preceding proposition immediately
entails:
\begin{prop}\label{giveclsd}
If a totally disconnected,
locally compact group $G$ admits an injective,
continuous homomorphism into $\GL_n(\K)$
for some $n\in \N$ and some local field~$\K$,
then $U_\alpha$ is closed in~$G$
for each inner automorphism~$\alpha$
of~$G$.\,\Punkt
\end{prop}
\begin{rem}
In particular, $U_\alpha$ is closed
for each inner automorphism
of a closed subgroup $G$ of the
general linear group $\GL_n(\K)$
over a local field~$\K$,
as already observed in~\cite[Remark~3.33\,(3)]{BaW}.
Also an analogue of~(\ref{hasanag})
can already be found in~\cite[Proposition~3.23]{BaW},
in the special case of Zariski connected
reductive linear algebraic groups
defined over~$\K$.
While our approach is analytic,
the special case is discussed in~\cite{BaW}
using methods
from the theory of linear algebraic groups.
\end{rem}
\begin{rem}\label{nonlinea}
Let $\K=\F(\!(X)\!)$, $G$ and its automorphism~$\alpha$
be as in~\ref{ncgp},
and $H:=G \rtimes \langle\alpha\rangle$.
Then $H$ is a $2$-dimensional
$\K$-analytic Lie group,
and conjugation with $\alpha$
restricts to $\alpha$ on~$G$.
In view of
the considerations in~\ref{ncgp},
we have found an inner automorphism of~$H$
with a non-closed contraction group.
By Proposition~\ref{giveclsd},
$H$ is not a linear Lie group,
and more generally it does not admit a faithful
continuous linear representation over a local
field.
\end{rem}
\noindent
For recent studies of linearity
questions on the level of pro-$p$-groups,
the reader is referred to \cite{ANS}, \cite{BaL}, \cite{CdS},
\cite{JZ}, \cite{JZK}, \cite{LaS}
and the references therein.
\section{The structure of contraction groups
in the case of positive characteristic}\label{sec11}
Wang's structural result concerning $p$-adic
contraction groups (as recalled in\linebreak
Theorem~\ref{Wa2})
can partially be adapted to Lie groups over local fields
of
positive characteristic
(see Theorems~A and~B in \cite{CPO}):
\begin{thm}\label{contraposi}
Let $G$ be an analytic Lie group
over a local field~$\K$ of positive
characteristic.
If~$G$ admits an analytic automorphism $\alpha\colon G\to G$
which is contractive,
then $G$ is a torsion group
of finite exponent.
Moreover, $G$ is nilpotent
and admits a central series
\[
\one\, =\, G_0\, \triangleleft \, G_1 \, \triangleleft \,\cdots\,
\triangleleft \, G_n \, =\, G
\]
where each $G_j$ is a Lie subgroup
of~$G$.
\end{thm}
\noindent
{\bf Some ideas of the proof.}
Let $p:=\car(\K)$.
It is known from~\cite{SIM}
that $G=\tor(G)\times H_{p_1}\times\cdots
\times H_{p_m}$
with certain $\alpha$-stable
$p_j$-adic Lie groups~$H_{p_j}$,
for suitable primes $p_1,\ldots, p_m$.
Since~$G$ is locally pro-$p$,
we must have $G=\tor(G)\times H_p$.
If $H_p$ was non-trivial,
then the size of the sets
of $p^k$-th powers of balls
in the $p$-adic Lie group~$H_p$
and such in~$G$ would differ
too much as~$k$ tends to~$\infty$,
and one can reach a contradiction
(if one makes this idea more precise,
as in the proof of \cite[Theorem~A]{CPO}).
Hence $G=\tor(G)$ is a torsion group.\\[2.1mm]
To see that $G$ is nilpotent,
we exploit that the $a$-stable
submanifolds $W^s_a(\alpha,1)$ are
Lie subgroups for all $a\in \;]0,1[\, \setminus R(\alpha)$
(see \cite[Proposition 6.2]{STF}).
Since
\[
[W^s_a(\alpha,1),W^s_b(\alpha,1)]\;\sub\;
W^s_{ab}(\alpha,1)
\]
holds for the commutator subgroups
whenever $a,b, ab\in
]0,1[\, \setminus R(\alpha)$
(as follows from
the second order Taylor expansion of
the commutator map),
one can easily
pick numbers $a_1<\cdots< a_n$
in $\,]0,1[\, \setminus R(L(\alpha))$
for some~$n$,
such that the Lie subgroups
$G_j:=W^s_{a_j}(\alpha,1)$
(for $j\in \{1,\ldots, n\}$)
form the desired
central series.\,\Punkt
\section{Further results that can be adapted
to positive characteristic}\label{sec12}
We mention that a variety of further
results can be generalized
from $p$-adic Lie groups
to Lie groups over arbitrary local fields,
if one assumes that the relevant
automorphisms have closed contraction groups.
The ultrametric inverse function theorems
(Propositions~\ref{invfct} and~\ref{invpar})
usually suffice as a replacement
for the
naturality of $\exp$
(although we cannot linearize,
at least balls are only deformed
as by a linear map).
And Theorem~\ref{Wa3} gives control
over the eigenvalues of~$L(\alpha)$.\\[2.1mm]
We now state two
results which can be generalized to
positive characteristic\linebreak
following this pattern.
\begin{numba}\label{unspec}
To discuss these results,
introduce more
terminology.
First, we recall from \cite{Pal}
that
a totally disconnected,
locally compact group~$G$
is called \emph{uniscalar}
if $s_G(x)=1$ for each $x\in G$.
This holds if and only if each group element
$x\in G$ normalizes some compact, open subgroup~$V_x$
(which may depend on~$x$).
It is natural to ask
whether this condition implies
that $V_x$ can be chosen independently
of $x$, i.e., whether $G$ has a compact, open,
\emph{normal} subgroup.
A suitable $p$-adic Lie group
(see \cite[\S6]{GW1})
shows that a positive answer can only be expected
if $G$ is compactly generated.
But even in the compactly generated
case, the answer is negative in general
(see \cite{BaM} and preparatory work in \cite{KaW}),
and one has to restrict attention
to particular classes of groups
to obtain a positive answer
(like compactly generated $p$-adic Lie groups).
If $G$ has the (even stronger)
property that every identity neighbourhood
contains an open, compact, normal subgroup
of~$G$, then $G$ is called \emph{pro-discrete}.
Finally, a Lie group~$G$ over a local field
is \emph{of type~$R$}
if the eigenvalues of $L(\alpha)$
have absolute value~$1$,
for each inner automorphism~$\alpha$
(see~\cite{Ra1}).
\end{numba}
\noindent
Using the above strategy,
one can generalize
results from~\cite{Ra1}
and \cite{GW1} (see~\cite{SPO}):
\begin{prop}
Let $G$ be an analytic Lie group
over a local field~$\K$,
and~$\alpha$ be an analytic automorphism
of~$G$.
Then the following properties
are equivalent:
\begin{itemize}
\item[\rm(a)]
$U_\alpha$ is closed,
and $s_G(\alpha)=s_G(\alpha^{-1})=1$;
\item[\rm(b)]
All eigenvalues of $L(\alpha)$
in~$\wb{\K}$
have absolute value~$1$;
\item[\rm(c)]
Every identity neighbourhood
of~$G$ contains a compact, open subgroup
$U$ which is $\alpha$-stable,
i.e., $\alpha(U)=U$.
\end{itemize}
In particular, $G$ is of type~$R$
if and only if $G$ is uniscalar
and $U_\alpha$ is closed
for each inner automorphism~$\alpha$ of~$G$
$($in which case $U_\alpha=\{1\})$.\,\Punkt
\end{prop}
\noindent
Using a fixed point theorem for group
actions on buildings by A. Parreau,
it was shown in \cite{Par}
(and \cite{GW1}) that every
compactly generated,
uniscalar $p$-adic Lie group
is pro-discrete.
More generally, one can prove (see~\cite{SPO}):
\begin{prop}
Every compactly generated
analytic Lie group of type~$R$ over
a local field is pro-discrete.\,\Punkt
\end{prop}
\noindent
As mentioned in~\ref{unspec},
results of the preceding form cannot be expected for general
totally disconnected, locally compact groups.
The author is grateful to the referee for the
following additional example (cf.\ also \cite{BaM} and \cite{KaW}).
Consider the Grigorchuk group~$H$.
Recall that~$H$ is a certain infinite (discrete) group
which is a $2$-group (in the sense that
every element has order a power of~$2$)
and finitely generated (by four elements).
It is known that~$H$ admits a transitive
left action
\[
H \times \Z\to \Z,\quad (h,n)\mto h.n
\]
on the set of integers such that each $h\in H$
commensurates $\N$ (in the sense that $h.\N$ and $\N$ have
finite symmetric difference);
such an action exists since~$H$ admits Schreier graphs
with several ends. The preceding action induces a left
action\footnote{for $(a_n)_{n\in \Z}\in \F_2^\Z$ with support bounded below.}
\[
\F_2(\!(X)\!)\to\F_2(\!(X)\!),\quad
\sum_{n\in \Z}a_nX^n\mto \sum_{n\in \Z}a_nX^{h.n}\vspace{-1mm}
\]
by automorphisms of the topological group $(\F_2(\!(X)\!),+)$,
endowed with its usual topology (such that $\F_2[\![X]\!]$ is a compact
open subgroup).
We use this action to form the semi-direct product
$G:=\F_2(\!(X)\!)\rtimes H$.
\begin{prop}
$G:=\F_2(\!(X)\!)\rtimes H$
is a compactly generated, totally disconnected
locally compact torsion group.
For every $g\in G$ and identity neighbourhood
$W\sub G$, there exists a compact open subgroup~$V$ of~$G$
such that~$V\sub W$ and~$V$ is normalized by~$g$;
moreover, the contraction group $U_{\alpha_g}$
of the inner automorphism $\alpha_g\colon G\to G$, $x\mto gxg^{-1}$
is trivial. However, $G$ does not have a compact, open,
normal subgroup; in particular, $G$ is not pro-discrete.
\end{prop}
\begin{proof}
After shrinking~$W$, we may assume that $W$ is a compact open subgroup
of~$G$. Since $\F_2(\!(X)\!)$ and $G/F_2(\!(X)\!)\cong H$ are torsion
groups, also~$G$ is a torsion group. If $m\in \N_0$ is the order
of $g\in G$, then $V:=\bigcap_{n=0}^mg^n(W)$ is a compact
open subgroup of~$G$ which is normalized by~$g$ and contained
in~$W$. If $e\not=x\in G$, after shrinking~$W$ we may
assume that $x\not\in W$ and hence $x\not\in V$.
As the complement of~$V$ in~$G$ is stable under
conjugation with~$g$, we deduce that $g^kxg^{-k}\not\in V$
for all $k\in \N$ and hence $g\not\in U_{\alpha_g}$,
whence $U_{\alpha_g}=\{e\}$ is trivial.
Suppose we could find a compact, open, normal subgroup $V\sub G$.
Then $X^n\F_2[\![X]\!]\sub V$ for some $n\in \Z$
and thus $X^n\in V$. Since $H$ acts transitively on~$\Z$,
we deduce that $X^m\in V$ for each $m\in \Z$
and thus $\F_2(\!(X)\!)\sub V$, as~$V$ is a closed subgroup of~$G$.
But then $\F_2(\!(X)\!)$ would be a closed subgroup of~$V$
and so $\F_2(\!(X)\!)$ would be compact, a contradiction.
\end{proof}
{\bf Helge  Gl\"{o}ckner}, Universit\"at Paderborn, Institut f\"{u}r Mathematik,\\
Warburger Str.\ 100, 33098 Paderborn, Germany;\\[1mm]
e-mail: {\tt  glockner@math.upb.de}\vfill
\end{document}